\pdfoutput=1
\documentclass[12pt]{paper}

\usepackage{color}
\usepackage{mathptmx}       % selects Times Roman as basic font
\usepackage{helvet}         % selects\textbf{\textbf{•}} Helvetica as sans-serif font
\usepackage{courier}        % selects Courier as typewriter font
\usepackage{type1cm}        % activate if the above 3 fonts are
\usepackage{makeidx}         % allows index generation
\usepackage{graphicx}        % standard LaTeX graphics tool
\usepackage{multicol}        % used for the two-column index
\usepackage[bottom]{footmisc}% places footnotes at page bottom
\usepackage{amsmath}
\usepackage{amsfonts}
\usepackage{amssymb}
\usepackage{amsthm}
\usepackage{amsopn}
\usepackage{subcaption}
\usepackage{comment}
\newtheorem{theorem}{Theorem}
\newtheorem{corollary}[theorem]{Corollary}
\newtheorem{condition}[theorem]{Condition}
\newtheorem{lemma}[theorem]{Lemma}

\newtheorem{proposition}[theorem]{Proposition}

\newtheorem{remark}[theorem]{Remark}
\newtheorem{definition}[theorem]{Definition}

\usepackage[top=2cm, bottom=2cm, left=2cm, right=2cm]{geometry}
\numberwithin{theorem}{section}
\numberwithin{figure}{section}
\numberwithin{equation}{section}
\DeclareMathOperator{\CR}{CR}
\DeclareMathOperator{\cov}{cov}
\DeclareMathOperator{\dist}{dist}
\DeclareMathOperator{\SLE}{SLE}

\DeclareMathOperator{\GFF}{GFF}

\begin{document}

\title{Level Lines of the Gaussian Free Field with general boundary data}
\author{Ellen Powell and Hao Wu}

% your contribution title if the original one is too long%

\maketitle

\abstract{We study the level lines of a Gaussian free field in a planar domain with general boundary data $F$. We show that the level lines exist as continuous curves under the assumption that $F$ is regulated (i.e., admits finite left and right limits at every point), and satisfies certain inequalities. Moreover, these level lines are a.s. determined by the field. This allows us to define and study a generalization of the SLE$_4(\underline{\rho})$ process, now with a continuum of force points. A crucial ingredient is a monotonicity property in terms of the boundary data which strengthens a result of Miller and Sheffield and is also of independent interest. }
%\tableofcontents
%\newpage
%\newcommand{\mathbb}{}
\newcommand{\eps}{\epsilon}
\newcommand{\ov}{\overline}
\newcommand{\U}{\mathbb{U}}
\newcommand{\T}{\mathbb{T}}
\newcommand{\HH}{\mathbb{H}}
\newcommand{\LA}{\mathcal{A}}
\newcommand{\LC}{\mathcal{C}}
\newcommand{\LD}{\mathcal{D}}
\newcommand{\LF}{\mathcal{F}}
\newcommand{\LK}{\mathcal{K}}
\newcommand{\LE}{\mathcal{E}}
\newcommand{\LL}{\mathcal{L}}
\newcommand{\LU}{\mathcal{U}}
\newcommand{\LV}{\mathcal{V}}
\newcommand{\LZ}{\mathcal{Z}}
\newcommand{\R}{\mathbb{R}}
\newcommand{\C}{\mathbb{C}}
\newcommand{\N}{\mathbb{N}}
\newcommand{\Z}{\mathbb{Z}}
\newcommand{\E}{\mathbb{E}}
\newcommand{\PP}{\mathbb{P}}
\newcommand{\QQ}{\mathbb{Q}}
\newcommand{\A}{\mathbb{A}}
\newcommand{\bn}{\mathbf{n}}
\newcommand{\MR}{MR}
\newcommand{\cond}{\,|\,}
\newcommand{\la}{\langle}
\newcommand{\ra}{\rangle}
\newcommand{\tree}{\Upsilon}

\section{Introduction}\label{sec::intro}
The relationship between Schramm--Loewner Evolution (SLE) and the two-dimensional Gaussian free field (GFF) is at the heart of recent breakthroughs in Liouville quantum gravity, imaginary geometry and more generally, random conformal geometry. Starting with the seminal papers of  \cite{dub}, \cite{ss}, \cite{ss09}, one key idea is to make sense of SLE-type curves as a level lines of an underlying Gaussian free field $h$ in a domain, which we take to be the upper half plane $\HH$ without loss of generality in the rest of the paper. When the field $h$ is given the boundary values $\lambda:=\pi/2$ on $\R_+$ and $-\lambda$ on $\R_-$, the corresponding level line is a chordal SLE$_4$ curve.
A considerable extension of that theory is described in \cite{msig1}, which introduced the notion of \emph{flow lines} and \emph{counter flow lines} of the GFF. In this case it turns out that the curves are given by $\SLE_\kappa$ processes with $\kappa\in (0,4)$ and $\kappa\in (4,\infty)$ respectively.  

It is also natural to wonder for which sort of boundary data the notion of level line makes sense. In \cite{msig1} and \cite{wwll1}, the hypothesis on the boundary data is extended from the above to any arbitrary piecewise constant function on the real line. The goal of this paper will be to relax this assumption. Assuming solely that the boundary data $F$ is a \emph{regulated} function, i.e., the left and right limits
\begin{equation}
\label{eqn::regulated}
F(t^+) = \lim_{h \to 0+} F(t+h) ; \quad F(t^-) = \lim_{h \to 0-} F(t+h)
\end{equation}
exist and are finite for all $t \in \overline{\R}$, and that for some $c>0$
\begin{equation}
\label{eqn::inequalities}
F(x)\le \lambda-c,\quad x<0;\quad F(x)\ge -\lambda+c,\quad x\ge 0
\end{equation}
 which roughly corresponds to the non existence of a continuation threshold, we can show that the corresponding level line is well defined almost surely as a continuous transient curve. Moreover, it is almost surely determined by the field. 

This also allows us, for a zero boundary GFF $h$, to consider the set of level lines of different heights. By this we mean the level lines of $h+F$, where $F$ ranges over (the bounded harmonic extensions of) all regulated functions on $\R$. Strengthening the results of \cite{msig1}, \cite{wwll1}, we are able to prove a general monotonicity principle for the level lines, which is both a key tool in our existence proof, and an interesting result in its own right. This is deeply intertwined with the reversibility property of the level lines, which we are also able to prove in general; see Theorems \ref{thm::monotonicity} and \ref{thm::reversibility}.
\begin{comment}
%Another extension of the theory, developed in \cite{msig1}, \cite{wwll1}, is to consider GFFs with different boundary data; given by arbitrary piecewise constant functions on the real line. In this case the marginal laws of the level lines and flow lines are given by $\SLE_\kappa(\underline{\rho})$ processes, where $\underline{\rho}$ is a vector corresponding to the boundary data. These are similar to $\SLE_\kappa$ curves, but we now keep track of certain \emph{force points} on the real line, which cause some attraction or repulsion of the curve from the boundary depending on the vector $\underline{\rho}$. With this theory in hand, we can consider for a zero boundary condition GFF, the collection of level lines of different heights, or flow lines of different angles. These satisfy certain monotonicity, merging and reversibility properties in their interactions, which correspond to what one would expect from level lines or flow lines of a smooth function.

%A natural question to ask is whether the requirement that the level lines have piecewise constant boundary data can be relaxed. The goal of the current paper will be to answer this, and we will show that if we assume the boundary data can be uniformly approximated by piecewise constant functions, and satisfies certain bounds, then this is indeed the case.  We will also show that these level lines are given by continuous transient curves, are determined by the field, and display the same interactions as in the piecewise constant boundary data case. 
\end{comment}

A further point of interest is that we obtain some continuity in the level lines as a consequence of our proof. That is, if we take a sequence of piecewise constant functions $F_n$ converging monotonically uniformly to some $F$, then the level lines of height $F_n$ for a zero boundary GFF converge almost surely to the level line of height $F$. This convergence is with respect to Hausdorff distance, after conformally mapping everything to the unit disc. 

We remark that our hypothesis on the boundary data is satisfied by a wide range of functions, including the special class of \emph{functions of bounded variation}. Any such function can be described almost everywhere as the integral of a finite Radon measure $\rho$, and this connection allows us to deduce that the marginal law of a level line with such boundary data is given by what we call an $\SLE_4(\rho)$ process. This is the natural analogue of an $\SLE_4(\underline{\rho})$ process, where the vector $\underline{\rho}$ is replaced by a measure. Our results therefore demonstrate the existence of such processes, as well as establishing some further properties. 

We first recall the definition of what it means for a curve, and more generally a Loewner chain, to be a level line. If we have a Loewner chain $(K_t, t\geq 0)$ in $\HH$, with associated sequence of conformal maps $g_t: \HH\setminus K_t \to \HH$, we will often want to describe the \emph{image under $g_t$} of a point $x$ on the real line. To do this, for any $x\leq 0$ we define a process $V_t^L(x)$ by setting it equal to $g_t(x)$ if $x\notin K_t$ and if $x\in K_t$, taking it to be the image of the leftmost point of $\R\cap K_t$ under $g_t$. We define a process $V_t^R(x)$ for $x\geq 0$ analogously. The process $V_t^{L}(x)$ for $x\in \R_-$, or $V_t^R(x)$ for $x\in \R_+$, is what we define to be the \emph{image of $x$ under $g_t$}.

\begin{definition}[\cite{msig1, wwll1}]
\label{def::gff_levelline}
Suppose that $F$ is $L^1$ with respect to harmonic measure on $\R$ viewed from some point in $\HH$ and that $h$ is a zero boundary $\GFF$ in $\HH$. 
If $(K_t, t\geq 0)$ is a Loewner chain and $(g_t, t\geq 0)$ is the corresponding sequence of conformal maps, set $f_t=g_t-W_t$, and let $V_t^R(x)$ (resp. $V_t^L(x)$) be the image of $x\ge  0$ (resp. $x\leq 0$) under $g_t$.  
Let $\eta_t^0$ be the bounded harmonic function on $\HH$ with boundary values (see Figure \ref{fig::intro_coupling})
\[\begin{cases}
F(f_t^{-1}(x)), &\text{if }x\ge V_t^R(0^+)-W_t,\\
\lambda,&\text{if } 0\le x<V_t^R(0^+)-W_t,\\
-\lambda, &\text{if } V_t^L(0^-)-W_t\le x<0,\\
F(f_t^{-1}(x)), &\text{if } x< V_t^L(0^-)-W_t.
\end{cases}\] 
Define, for $z\in\HH\setminus K_t$,
\[\eta_t(z)=\eta_t^0(f_t(z)).\]

We say that $K$ is a level line of $h+F$ if there exists a coupling $(h, K)$ such that the following domain Markov property holds: for any finite $K$-stopping time $\tau$, given $K_{\tau}$, the conditional law of $(h+F)|_{\HH\setminus K_{\tau}}$ is equal to the law of $h\circ f_{\tau}+\eta_{\tau}$. 

\end{definition}
  
\begin{figure}[ht!]
\begin{center}
\includegraphics[width=0.7\textwidth]{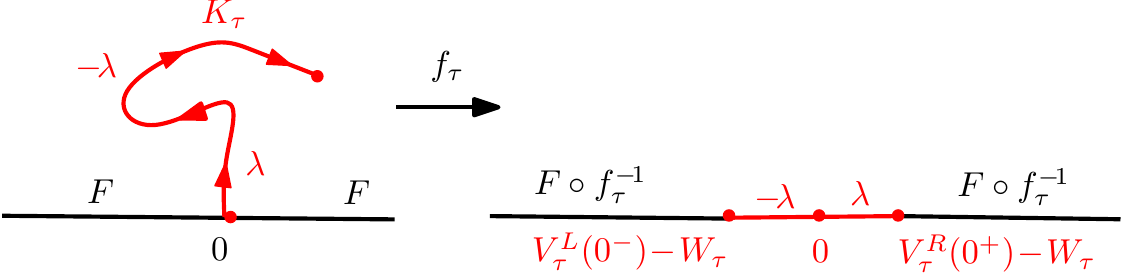}
\end{center}
\caption{\label{fig::intro_coupling} The left hand side shows the boundary values of the harmonic function $\eta_\tau$ in $\HH\setminus K_\tau$. This is the image under $f_\tau^{-1}$ of the harmonic function $\eta_\tau^0$ in $\HH$, whose boundary values are shown on the right hand side.}
\end{figure}

Note that this definition is the same for any two functions $F_1$ and $F_2$ which are equal almost everywhere, since the harmonic extensions of such functions are necessarily equal. 
From Definition \ref{def::gff_levelline}, we can see that the so-called level lines of the $\GFF$ have an intriguing property that distinguishes them from level lines of an ordinary smooth function. Namely, once one conditions on a level line, the conditional expectation of the field on one side of the curve differs by $2\lambda$ from the value on the other side. In a sense, a level line is more like a ``level cliff" where there is a prescribed jump between the two sides of the curve.
\medbreak

More generally, we say that a Loewner chain $(K_t,t\geq 0)$ is a level line of a GFF $h$ in a domain $D$ from $a\in \partial D$ to $b\in \partial D$ if $(\varphi(K_t),t\geq 0)$ is a level line of $\varphi(h)$ as in Definition \ref{def::gff_levelline}, where $\varphi$ is a conformal map from $D\to \HH$ sending $a$ to $0$ and $b$ to $\infty$.

\begin{theorem}\label{thm::gff_levelline_coupling}
[Coupling]
Assume the same notations as in Definition \ref{def::gff_levelline}. Suppose that the function $F$ is regulated and satisfies (\ref{eqn::inequalities}) for some $c>0$. 
Then there exists a coupling satisfying the conditions in Definition \ref{def::gff_levelline}. Moreover, in this coupling, the Loewner chain $K$ is almost surely generated by a continuous and transient curve $\gamma$ with almost surely continuous driving function.
\end{theorem}

The inequality on $F$ in Theorem \ref{thm::gff_levelline_coupling} guarantees that the corresponding level line will reach its target point $\infty$ before ``dying" at some continuation threshold. Indeed, the level line of a GFF with piecewise constant boundary data is only defined until the first time that it hits a section of $\R_+$ where the boundary data is less than $-\lambda$ or a section of $\R_-$ where it is greater than $\lambda$. In our case, if we allowed $F$ to approach $-\lambda$ (resp. $\lambda$) at some point in $\R_+$ (resp. in $\R_-$), then our current framework would not control the behaviour of the level line around this point (see discussion below.) Thus, we do not treat this situation here.

\begin{theorem}\label{thm::gff_levelline_determination}[Determination]
If $(h,\gamma)$ are coupled as in Theorem \ref{thm::gff_levelline_coupling}, then $\gamma$ is almost surely determined by $h$. Moreover, the curve $\gamma$ is almost surely simple. 
We call $\gamma$ the level line of $h+F$.
\end{theorem}

With this in hand, we can consider the collection of level lines determined by a given field. The following two theorems describe the interactions between the curves; corresponding to what one might expect from the level lines of a smooth function. 

\begin{theorem}\label{thm::monotonicity}
[Monotonicity]
Suppose that $F, G$ are functions satisfying the conditions in Theorem \ref{thm::gff_levelline_coupling}, and that $F(x)\ge G(x)$ for $x\in\R$. Suppose that $h$ is a zero boundary $\GFF$ on $\HH$ and $\gamma_F$ (resp. $\gamma_G$) is the level line of $h+F$ (resp. $h+G$). Then $\gamma_F$ lies to the left of $\gamma_G$ almost surely.
\end{theorem}

\begin{theorem}\label{thm::reversibility}
[Reversibility]
Suppose that $h$ is a $\GFF$ on $\HH$ whose boundary value satisfies the conditions in Theorem \ref{thm::gff_levelline_coupling}. Let $\gamma$ be the level line of $h$ from $0$ to $\infty$ and $\gamma'$  be the level line of $-h$ from $\infty$ to $0$. Then the two paths $\gamma$ and $\gamma'$ (viewed as sets) are equal almost surely. 
\end{theorem}

Now we will explain the relevance of Conditions (\ref{eqn::regulated}) and $(\ref{eqn::inequalities})$, which we need for our approach to work. Although one can make sense of what it means to be a level line of $h+F$ for any $F$ in $L^1$ (as in Definition \ref{def::gff_levelline}), before this work the existence of the coupling was only known for piecewise constant boundary data. 
The assumption that the boundary data $F$ is regulated corresponds precisely to the fact that $F$ can be uniformly approximated by piecewise constant functions. Indeed, our argument will use an approximation of $F$ by such functions, and a limit of the corresponding level lines. Thus with our current approach we are unable to say anything about functions which are not regulated. However, since Definition \ref{def::gff_levelline} still makes sense for a wider class of functions, it is an interesting question to determine the most general restrictions under which a coupling exists. For example, if one takes a GFF with boundary data which is $-\lambda$ in a neighbourhood to the left of $0$ and $\lambda$ in a neighbourhood to the right of $0$ then one can allow much rougher boundary data away from these neighbourhoods (for example, even Neumann boundary conditions, see \cite{hadamard}), and construct a weaker form of ``local coupling" with an $\SLE$ variant. Whether these types of coupling can be extended to a strong coupling as in Definition \ref{def::gff_levelline}, where the curve is also determined by the field, or whether the condition near $0$ can be relaxed is currently unknown.

Concerning Condition (\ref{eqn::inequalities}); the key to the proof of Theorem \ref{thm::gff_levelline_coupling} is the continuity and transience of the approximating level lines (with piecewise constant boundary data). This allows us to use the results of \cite{ksrc} (see details in Section \ref{subsec::cvg_curves}) to obtain a continuous limiting curve. If Condition (\ref{eqn::inequalities}) failed, the approximating level lines would only be defined up to a continuation threshold, and we would not be able to obtain such a limit. The continuity of the limiting curve is absolutely crucial to the proofs of Theorems \ref{thm::gff_levelline_determination} to \ref{thm::reversibility}. In fact, if the existence and the continuity of level lines were obtained for other boundary data, one could use similar proofs to get the corresponding theorems. However, whether continuity still holds in this set up is also a difficult open problem. Although it is natural to conjecture that for general regulated boundary data the level line will exist as a continuous curve until hitting a point on the boundary where Condition (\ref{eqn::inequalities}) fails, a ``continuation threshold" as in \cite{msig1},\cite{wwll1}, it is unclear whether or not the continuity will break down around this point. 
\medbreak
Finally, we identify the law of the level lines. It is proved in \cite{msig1, wwll1} that the level lines of GFF with piecewise constant boundary data are $\SLE_4(\underline{\rho})$ processes where $\underline{\rho}$ is a vector. In our context, when the boundary data is of bounded variation, the level lines turn out to be  $\SLE_4(\rho)$ processes, where $\rho$ is now a Radon measure. With the help of the GFF, we are able to obtain the existence, the continuity, and the reversibility of such processes, properties which are far from clear by the definition of the process through Loewner evolution.

\begin{theorem}\label{thm::sle4rho_existence}
Assume the same notations as in Theorem \ref{thm::gff_levelline_coupling}. Suppose further that $F$ is of bounded variation. Then in the coupling $(h,\gamma)$ given by Theorem \ref{thm::gff_levelline_coupling}, the marginal law of $\gamma$ is that of an $\SLE_4(\rho^L;\rho^R)$ process (see Section \ref{subsec::slerho}) where $\rho^R$ (resp. $\rho^L$) is a finite Radon measure on $\R_+$ (resp. on $\R_-$)  and
\[\begin{cases}
F(x)=\lambda(1+\rho^R([0,x])),& x\ge 0;\\
F(x)=-\lambda(1+\rho^L((x,0])), & x<0
\end{cases}\]
almost everywhere. In particular, we have the following properties of the $\SLE_4(\rho^L;\rho^R)$ process. Suppose that there exists $c>0$ such that 
\[\rho^L((x,0])\ge -2+c,\quad x<0,\quad \rho^R([0,x])\ge -2+c,\quad x> 0.\] Then
\begin{enumerate}
\item [(1)] There exists a law on continuous curves from $0$ to $\infty$ in $\overline{\HH}$ with almost surely continuous driving functions, for which the associated Loewner chain is an ar$\SLE_4(\rho^L;\rho^R)$ process. 
\item [(2)] The above continuous curve is almost surely simple and transient. 
\item [(3)] The time reversal of the above $\SLE_4(\rho^L;\rho^R)$ process has the same law as $\SLE_4(\tilde{\rho}^L;\tilde{\rho}^R)$ ,where 
\[\tilde{\rho}^R([x,\infty])=\rho^L((x,0]),\quad x<0; \quad \tilde{\rho}^L((x,\infty])=\rho^R([0,x]),\quad  x>0.\]
\end{enumerate}
\end{theorem}

\begin{remark}
Although Theorem \ref{thm::sle4rho_existence} gives us existence of $\SLE_4(\rho^L;\rho^R)$ processes, we do not derive uniqueness in law. That is, we have not excluded the possibility that there exists another law on Loewner chains satisfying the definition of an $\SLE_4(\rho^L;\rho^R)$ process.
\end{remark}
\begin{remark}
Item (3) is the so-called reversibility of $\SLE$. The reversibility was derived previously for $\SLE_{\kappa}$ in \cite{ZhanReversibilityChordalSLE}, for $\SLE_{\kappa}(\rho)$ where $\rho$ is a vector in \cite{ZhanDualityChordalSLE, MillerSheffieldIG2, MillerSheffieldIG3, WernerWuCLESLE}. In Theorem \ref{thm::sle4rho_existence}, we derive the reversibility of $\SLE_4(\rho)$ where $\rho$ is a Radon measure.
\end{remark}
\medbreak
\noindent\textbf{Outline.} The structure of the paper is as follows. In Section \ref{sec::preliminaries}, we discuss briefly the necessary background theory, and collect some results that will be important to us. We also define the class of $\SLE_\kappa(\rho)$ process and generalize some of the theory from \cite{msig1}, \cite{wwll1} which will help us in the sequel. In Sections \ref{sec::nonboundary_intersecting} and \ref{sec::monotonicity}, we set up a general framework for the level lines of a GFF, under the assumption that they exist and are given by continuous transient curves. In particular, we show that they are monotonic in the boundary data, and describe where they can and cannot hit the boundary. Sections \ref{sec::weakcvg} and \ref{sec::existence} address the existence of continuous transient curves which can be coupled as level lines of a GFF, provided the boundary data satisfies the conditions of Theorem \ref{thm::gff_levelline_coupling}. The proof of this is via an approximation argument; using a general theory for the weak convergence of curves, as set out in \cite{ksrc}. The key point in the proof is the monotonicity obtained in Section \ref{sec::monotonicity}. In Section \ref{sec::reversibilityetc} we prove Theorems \ref{thm::gff_levelline_determination} to \ref{thm::reversibility} using the ideas from Sections \ref{sec::nonboundary_intersecting} and \ref{sec::monotonicity}. Finally, we complete the proof of Theorem \ref{thm::sle4rho_existence} in Section \ref{sec::sle4rho}.  
\medbreak
\noindent\textbf{Acknowledgments.} We thank Nathana\"{e}l Berestycki, Jason Miller, Steffen Rohde, and Scott Sheffield for helpful discussions. 
We thank Avelio Sepúlveda and Juhan Aru for precious comments on the previous version of this paper. 
The main part of this work was done while H.\ Wu was at MIT and 
H.\ Wu's work is funded by NSF DMS-1406411. E.\ Powell's is funded by a Cambridge Centre for Analysis EPSRC studentship.

\section{Preliminaries}\label{sec::preliminaries}
\subsection{Regulated functions and functions of bounded variation}

We say that a function $F$ on $\R$ is \emph{regulated} if it admits finite left and right limits 
\[F(t^+) = \lim_{h \to 0+} F(t+h) ; \quad F(t^-) = \lim_{h \to 0-} F(t+h)\]
at every point $t\in \R$, including $\infty$. Equivalently, see \cite[Secion 7.6]{jd}, $F$ is regulated if it can be uniformly approximated on $\R$ by piecewise constant functions which change value only finitely many times. It is this formulation of the definition that will be useful to us in the sequel.

Another type of function which is of particular interest in the current paper is the class of functions of \emph{bounded variation}. Let us consider the connection (\ref{eqn::frho}) between pairs of Radon measures $(\rho^L;\rho^R)$ and functions $F$ on the real line. We saw above that piecewise constant functions correspond to purely atomic measures. In general, finite Radon measures are in one-to-one correspondence with right-continuous functions of bounded variation.

The space of functions of bounded variation are those $F$ which satisfy
\[\sup_{a<b} \left( \sup \left\{  \sum_i |F(x_i)-F(x_{i-1})| \; : \; \{x_i\} \; \text{a finite partition of} \; [a,b] \right\} \right)< \infty. \]
For a proof of this equivalence, see \cite[Theorem 3.29]{realanalysis}. Note that these functions are clearly regulated. So, provided they satisfy the correct bounds on $\R_-$ and $\R_+$, functions of bounded variation meet the conditions of Theorem \ref{thm::gff_levelline_coupling}.

%These functions are also equivalent, see \cite[Theorem 3.27]{realanalysis}, to functions which can be written as the difference of two bounded non-decreasing functions. In particular, they are regulated.

Furthermore, if a bounded variation function is also absolutely continuous, then the corresponding measures $(\rho^L;\rho^R)$ are absolutely continuous with respect to Lebesgue measure, and writing \[\rho^L(dx)= f^L(x) \, dx, \quad \rho^R(dx)=f^R(x) \, dx\] we have that the function is differentiable almost everywhere with derivative equal to $f^L(x)$ on $\R_-$ and $f^R(x)$ on $\R_+$. 

\subsection{A result on the convergence of curves}\label{subsec::cvg_curves}
To show existence of the level line of a $\GFF$ with general boundary data as given in Theorem \ref{thm::gff_levelline_coupling}, we will attempt to approximate it by level lines of the field with piecewise constant boundary data. For this, a result from \cite{ksrc} on the weak convergence of curves, satisfying certain conditions on crossing probabilities, will be crucial. 

In order to state the result, we need to define what we mean by \emph{crossings of topological quadrilaterals.}

\begin{definition}\label{defn::topquad}
A \emph{topological quadrilateral} $Q=(V;S_k, k=0,1,2,3)$ consists of a domain $V$, along with four boundary arcs $S_0, S_1, S_2, S_3$, which can be mapped homeomorphically to a square in such a way that the boundary arcs are in counterclockwise order and correspond to the edges of the square. For any topological quadrilateral, there exists a unique positive $L$ and a conformal map from $Q$ onto the rectangle $[0,L]\times [0,1]$, such that the boundary arcs are mapped to the edges of the quadrilateral and, in particular, $S_0$ is mapped to $\{0\}\times [0,1]$. We call this unique $L$ the \emph{modulus} of $Q$, denoted by $m(Q)$.
\end{definition}

\begin{definition}\label{defn::quadcross}
We will often consider topological quadrilaterals in $\HH$ which lie on the boundary in the sense that $S_1, S_3\subset \R$ and $S_0,S_2 \subset \HH$. If we have such a quadrilateral, then we say that a curve $\gamma:[T_0,T_1]\to \C$ \emph{crosses} Q if there is a subinterval $[t_0,t_1]\subset [T_0,T_1]$, such that $\gamma(t_0,t_1)\subset V$ but $\gamma[t_0,t_1]$ intersects both $S_0$ and $S_2$. 
\end{definition}

Essentially, the condition that will be required for weak convergence will be the following: 
\begin{condition}\label{cond::quadcross}
For any simple curve $\gamma$ on $\HH$ we say that $Q$ is a topological quadrilateral in $H_\tau:=\HH\setminus \gamma[0,\tau]$ if it is the image of the square $(0,1)^2$ under a homeomorphism $\psi$. We define the sides of $Q$:
$S_0, S_1, S_2, S_3$, to be the images of 
$$ \{0\}\times (0,1), \;\;\; (0,1)\times \{0\}, \;\;\; \{1\}\times (0,1), \;\;\; (0,1)\times \{1\}$$
under $\psi$. We consider Q such that the opposite sides $S_1, S_3$ are contained in $\partial H_t$ and define a crossing of $Q$ to be a curve in $H_t$ which connects the two opposite sides $S_0$ and $S_2$. Finally, we say that $Q$ is \emph{avoidable} if it doesn't disconnect $\gamma(\tau)$ and $\infty$ inside $H_t$. 

A family $\Sigma $ of probability measures on simple curves from $0$ to $\infty$ in $\HH$ is said to satisfy a \emph{conformal bound on an unforced crossing} if there exists a constant $M>0$ such that for any $\PP\in \Sigma$, for any stopping time $\tau$, and any avoidable quadrilateral $Q$ of $H_\tau$ whose modulus $m(Q)$ is greater than $M$,
$$ \PP\left(\gamma[\tau, \infty) \, \text{crosses} \, Q \;|\; \gamma[0,\tau]\right) \leq 1/2. $$
\end{condition}

\begin{figure}[ht!]
\begin{subfigure}[b]{0.48\textwidth}
\begin{center}
\includegraphics[width=0.73\textwidth]{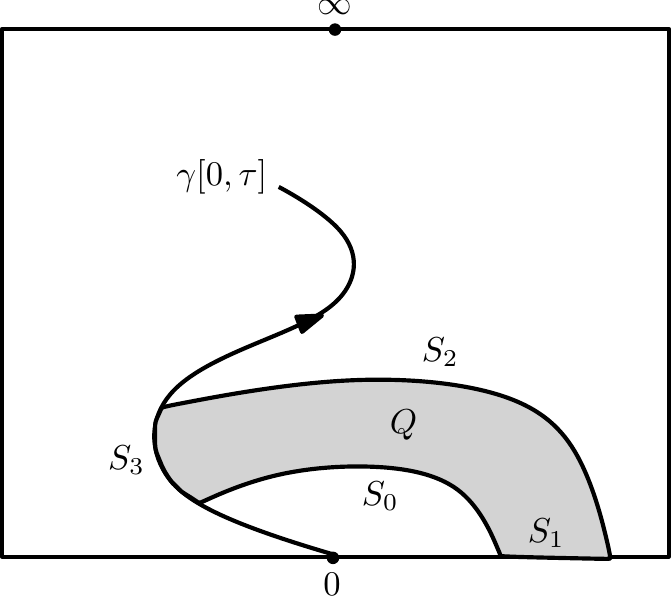}
\end{center}
\caption{An avoidable quadrilateral of $H_\tau=\HH\setminus \gamma[0,\tau]$.}
\end{subfigure}
$\quad$
\begin{subfigure}[b]{0.48\textwidth}
\begin{center}\includegraphics[width=0.73\textwidth]{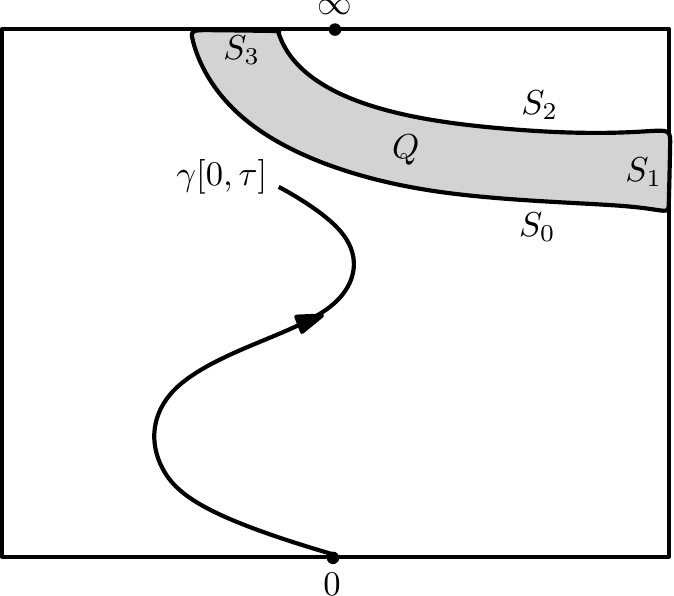}
\end{center}
\caption{An unavoidable quadrilateral of $H_\tau=\HH\setminus \gamma[0,\tau]$.}
\end{subfigure}
\caption{\label{fig::quadrilaterals}} 
\end{figure}
Now we may state the result. 

\begin{proposition}\label{propn::convergenceofcurves}
Suppose that $(W^{(n)})_{n\in \N}$ is a sequence of driving processes of random Loewner chains that are generated by continuous simple random cuves $(\gamma^{(n)})_{n\in \N}$ in $\HH$, satisfying Condition \ref{cond::quadcross}. Suppose that the $(\gamma^{(n)})_{n\in \N}$ are parameterized by half plane capacity. Then
\begin{itemize}
\item $(W^{(n)})_{n\in \N}$ is tight in the metrisable space of continuous functions on $[0,\infty)$ with the topology of uniform convergence on compact subsets of $[0,\infty)$. 
\item $(\gamma^{(n)})_{n\in \N}$ is tight in the metrisable space of continuous functions on $[0,\infty)$ with the topology of uniform convergence on the compact subsets of $[0,\infty)$. 
\end{itemize}
Moreover, if the sequence converges weakly in either of the topologies above, then it also converges weakly in the other and the limits agree in the sense that the law of the limiting random curve is the same as the that of the random curve generated under the law of the limiting driving process.  In particular, any subsequential limit of the sequence of curves a.s. generates a Loewner chain with continuous driving function.
\end{proposition}

\begin{proof}
This may be found in \cite{ksrc} cf. Theorem 1.5 and Corollary 1.7. 
\end{proof}

In fact, we will need to apply this theorem when the curves $(\gamma^{(n)})_{n\in \N}$ correspond to certain $\SLE_4(\underline{\rho}^L;\underline{\rho}^R)$ processes. In this case they may hit the real line, and so are not necessarily contained in $\HH$, as required by the Proposition. However, as discussed before the proof of Theorem 1.10 in \cite{ksrc}, the result extends to curves such as ours, and so we may apply it without concern.

\subsection{The zero boundary Gaussian free field}
\label{gffprelims}
In this section we will describe the zero boundary  Gaussian free field ($\GFF$) in an arbitrary domain $D\subsetneq \C$. We will always assume that the domain has harmonically non-trivial boundary, meaning that a Brownian motion started from a point in the interior will hit the boundary almost surely.

We start with the \emph{Green's function} $G_D$ in $D$, which is the unique function in $D$ such that 
\begin{itemize}
\item $\Delta G_D(z, \cdot) = 2\pi \delta_z (\cdot)$ for each $z\in D$, and 
\item $G_D(z,w)=0$ if $z$ or $w$ is in $\partial D$. 
\end{itemize}
Explicitly, $$G_D(z,w)=-\log|z-w|-\tilde{G}_z(w)$$ where $\tilde{G}_z(w)$ is the harmonic extension of $w\mapsto -\log |z-w|$ from $\partial D$ to $D$. The Green's function is conformally invariant in the sense that for any conformal map $\phi$ on $D$, and $z,w\in D$, we have  $$G_D(z,w)=G_{\phi(D)}(\phi(z),\phi(w)).$$

Roughly speaking, the $\GFF$ will be the random Gaussian ``function" on $D$ with $\cov (h(z),h(w))=G_D(z,w)$. However, it can only be made sense of rigorously as a random distribution on $D$. 
For $H_s(D)$ the space of smooth compactly supported functions on $D$, we let $(\cdot, \cdot)$ denote the normal $L^2$ inner product on $H_s(D)$. We may also endow $H_s(D)$ with the \emph{Dirichlet inner product} defined by 
\[(f,g)_{\nabla}= \frac{1}{2\pi} \int_D\nabla f(z)\cdot \nabla g(z) d^2 z \]
and we denote its Hilbert space completion under Dirichlet inner product by $H(D)$. 

For $\{\phi_n\}_{n\geq 0}$ an orthonormal basis of $H(D)$, we define the \emph{zero boundary  $\GFF$} $h$ to be the random sum 
$h := \sum_n \alpha_n \phi_n$,
where the $\alpha_n$'s are i.i.d. Gaussians with mean 0 and variance 1. This almost surely diverges in $H(D)$, but makes sense as a distribution. That is, the limit $\sum_n \alpha_n (\phi_n, p):= (h,p)$ almost surely exists for each $p\in H_s(D)$, and $p\mapsto (h,p)$ is almost surely a continuous linear functional on $H_s(D)$.
Note that for any $f\in H_s(D)$ we have that $-\Delta f=p$ is also in $H_s(D)$ and so can define
\[(h,f)_\nabla:= \frac{1}{2\pi} (h,p).\]
Then $(h,f)_\nabla$ is a Gaussian with mean 0 and variance 
$$\frac{1}{4\pi^2} \sum_n (\phi_n, p)^2 = \sum_n (\phi_n, f)_\nabla^2 = (f,f)_\nabla.$$ In fact, this characterizes the Gaussian free field.
Furthermore, noticing that for $p\in H_s(D)$ 
$$\Delta^{-1} p := \frac{1}{2\pi} \int_D G_D(\cdot, w)p(w) \, dw $$ is a smooth function in $D$ whose Laplacian is $p$ and vanishes on $\partial D$, we see that for any $f,g,p,q\in H_s(D)$  
\[\cov((h,f)_\nabla, (h,g)_\nabla) = (f,g)_\nabla,\quad \cov((h,p),(h,q))=\int \int_{D\times D} p(z)G_D(z,w)q(w) d^2z d^2w.\]

\begin{proposition}\label{prop::gff_markov}
[The Markov Property]
Let $W\subset D$ be open and $h$ be a zero boundary $\GFF$ on $D$. Then we can write 
$$h=h_1 + h_2$$
where $h_1$ and $h_2$ are independent, $h_1$ is harmonic in $W$, and $h_2$ is a zero boundary  $\GFF$ in $W$.
\end{proposition}
This tells us that, given $h|_{D\setminus W}$, the conditional law of $h|_W$ is that of a zero boundary  $\GFF$ in $W$, plus the harmonic extension of $h|_{D\setminus W}$ to $W$. 
\begin{comment}

\begin{proof}
If we let $H(W)\subset H(D)$ be the space of functions which are supported in $W$ and $H^\perp(W)\subset H(D)$ be those which are harmonic in $W$, then we have the orthogonal decomposition
$$H(D)=H(W)\oplus H^\perp(W).$$
This follows, as we can write any $f\in H(D)$ as $f=f_w + f_w^\perp$, where $f_w^\perp$ is the harmonic function on $W$ equal to $f$ on $D\setminus W$, and $f_w:=f-f_w^\perp$. Clearly $f_w\in H(W)$ and $f^\perp_w\in H^\perp(W)$, and furthermore, for any $f\in H(W), g\in H^\perp(W)$, $$(f,g)_\nabla = \frac{1}{2\pi} \int_D \nabla f \cdot \nabla g = -\frac{1}{2\pi} \int_D \Delta f g = 0.$$
Thus, letting $\{\phi_n^1\}_{n\geq 0}$ be an orthonormal basis for $H^\perp(W)$ and $\{\phi_n^2\}_{n\in \N}$ be an orthonormal basis for $H(W)$, we see that together they form an orthonormal basis for $H(D)$ and so may write 
$$h = \sum_n \alpha_n^1 \phi_n^1 + \sum_n \alpha_n^2 \phi_n^2 := h_1 + h_2. $$
$h_1$ and $h_2$ are clearly independent and, since $\{\phi_n^2\}_{n\geq 0}$ forms an orthonormal basis for $H(W)$, it follows that $h_2$ is a zero boundary $\GFF$ on $W$. Finally, we see from the definition of $h_1$ that it is almost surely harmonic on $W$ in the distributional sense, and then by elliptic regularity, that it is in fact almost surely harmonic in the usual sense. 
\end{proof} 
\end{comment}

Suppose that $F$ is $L^1$ with respect to harmonic measure on $\R$ viewed from some point (hence every point) in $\HH$; we also denote its bounded harmonic extension to $\HH$ by $F$. Then the \emph{$\GFF$ with mean F} is defined to be the sum, $h+F$, of a zero boundary  $\GFF$ and $F$.  

\begin{proposition}\label{prop::gffabscont} 
Suppose that $D_1$ and $D_2$ are two simply connected domains with non empty intersection, and $h_i$ is a zero boundary  $\GFF$ on $D_i$ for $i=1,2$. Let $F_i$ be harmonic on $D_i$, $i=1,2$ and $U\subset D_1 \cap D_2$ be a simply connected open domain. Then
\begin{enumerate}
\item [(1)] If $\dist(U,\partial D_i)>0$ for $i=1,2$, then the laws of $$(h_1+F_1)|_U \;\;\text{and} \;\; (h_2+F_2)|_U $$ are mutually absolutely continuous.
\item [(2)] Suppose there is a neighbourhood $\bar{U}\subset U'$ such that $D_1\cap U'=D_2\cap U'$ and that $F_1-F_2$ tends to 0 along sequences approaching points in $\partial D_i \cap U'$. Then the laws of $$(h_1+F_1)|_U \;\;\text{and} \;\; (h_2+F_2)|_U $$ are mutually absolutely continuous.
\end{enumerate}
\end{proposition}
\begin{proof}
\cite[Proposition 3.2]{msig1}.
\end{proof}
\begin{comment}
Fix a domain $D\subsetneq \C$ and a point $z\in D$. Then, for $h$ a zero boundary  $\GFF$, we can make sense of the \emph{circle average} of $h$ around the point $z$. 
For any $\epsilon>0$, set $$\xi_\epsilon^z(w)=-\log (\epsilon \vee |z-w|)-\tilde{G}_{z,\epsilon}(w)$$ where $\tilde{G}_{z,\epsilon}(w)$ is the harmonic extension of the function $w\mapsto \log(\epsilon \vee |z-w|)$ from $\partial D$ into $D$. One can check that $\xi_\epsilon^z\in H(D)$ and, as distributions, $-\Delta \xi_\epsilon^z = 2\pi \rho_\epsilon^z$, where $\rho_\epsilon^z$ is the uniform measure on $\partial B(z,\epsilon)$. We then define the circle average to be 
$$h_\epsilon(z):= (h, \xi_\epsilon^z)_\nabla.$$
Note that if $d(z,\partial D)\geq \epsilon$ we have that $\tilde{G}_{z,\epsilon}=\tilde{G}_z$ and so $(\xi_\epsilon^z, \xi_\epsilon^z)_\nabla=(\xi_\epsilon^z,\rho_\epsilon^z)=-\log \epsilon - \tilde{G}_z(z)$. Moreover, if $\phi$ is a conformal map from $D$ to the unit disc $\mathbb{D}$, then 
$$-\tilde{G}_z(w)=\log|z-w| + G_D(z,w)=\log|z-w|+G_\mathbb{D}(0,\phi(w))=-\log \left|\frac{\phi(w)}{z-w}\right|$$ by conformal invariance of the Green's function. Setting $\phi$ to be the map that sends $z$ to $0$ and has $\phi'(z)>0$ we have that $\phi'(z)=\CR(z,D)$, the \emph{conformal radius of $D$ seen from $z$}, and so by the above, that $\tilde{G}_z(z)=-\log \CR(z,D).$ 
Hence, $h_\epsilon(z)$ is a Gaussian with mean $0$ and variance $-\log \epsilon + \log \CR(z,D)$.
\end{comment}
\subsection{Local sets for the $\GFF$}

The theory of \emph{local sets} for the Gaussian free field was first  introduced by Schramm and Sheffield in \cite{ss}, and we quote several of their results here.
For $D$ a simply connected domain and $A$ a random closed subset of $\bar{D}$, we let 
\[A_\delta := \{z\in D: d(z,A)\leq \delta\}\] and $\mathcal{A}_\delta$ be the smallest $\sigma$-algebra for which $A$ and the restriction of $h$ to the interior of $A_{\delta}$ are measurable. Setting $\mathcal{A}=\cap_{\delta\in \QQ_+} \mathcal{A}_\delta$ we obtain a $\sigma$-algebra which is intuitively the smallest such making $A$, and $h$ restricted to some infinitesimal neighbourhood of $A$, measurable. With this in mind, we will often refer to $\mathcal{A}$ as $(A,h|_A)$. 

\begin{definition}
Suppose that $(h,A)$ is a coupling of a $\GFF$ in $D$ and a random closed subset $A\subset \overline{D}$. Then we say that $A$ is a \emph{local set} for $h$ if either of the following equivalent statements hold:
\begin{enumerate}
\item [(1)] For any deterministic open subset $U\subset D$ we have that, given the orthogonal projection of $h$ onto $h^\perp(U)$, the event $\{A\cap U =\emptyset \}$ is independent of the orthogonal projection of $h$ onto $H(U)$. This means that the conditional probability of $\{A\cap U=\emptyset \}$ given $h$ is a measurable function of the orthogonal projection of $h$ onto $H^\perp(U)$.
\item [(2)] Given $\mathcal{A}$, the conditional law of $h$ is that of $h_1+h_2$, for $h$ a zero boundary  $\GFF$ on $D\setminus A$ and $h_1$ an $\mathcal{A}$-measurable random distribution which is almost surely harmonic on $D\setminus A$.
\end{enumerate}
In this case, we let $\mathcal{C}_A$ be the conditional expectation of $h$ given $(A,h|_A)$, corresponding to $h_1$ in Item (2).
\end{definition}

The interactions between local sets display some nice properties, which we will describe in the following propositions.

\begin{proposition}\label{propn::localsetunions}
Suppose that $A_1$, $A_2$ are local sets for a $\GFF$ $h$, which are conditionally independent given $h$. Then $A=A_1\cup A_2$ is also local for $h$ and moreover, given $(A_1,A_2,A,h|_A)$, the conditional law of $h$ is given by $\mathcal{C}_A$ plus an instance of the zero boundary  $\GFF$ in $D\setminus A$.  
\end{proposition}
\begin{proof}
\cite[Lemma 3.10]{ss}. 
\end{proof}

\begin{proposition}\label{propn::localsetaverages}
Let $A_1,A_2$ be connected local sets which are conditionally independent and $A=A_1\cup A_2$. Then $\mathcal{C}_A-\mathcal{C}_{A_2}$ is almost surely a harmonic function in $D\setminus A$ which tends to zero along any sequence converging to a limit in
\begin{itemize}
\item a connected component of $A_2\setminus A_1$ which is larger than a singleton, or
\item a connected component of $A_1\cap A_2$ which is larger than a singleton, if the limit is at a positive distance from either $A_2\setminus A_1$, or $A_1\setminus A_2$.
\end{itemize}
\end{proposition}
\begin{proof}
\cite[Lemma 3.11]{ss} and \cite[Proposition 3.6]{msig1}.
\end{proof}

\begin{proposition}\label{propn::morelocalsets}
Let $A_1, A_2$ be connected local sets which are conditionally independent and $A=A_1\cup A_2$. Suppose that $C$ is a $\sigma(A_1)$-measurable connected component of $D\setminus A_1$ such that $\{C\cap A_2=\emptyset\}$ almost surely. Then $\mathcal{C}_A|_C=\mathcal{C}_{A_1}|_C$ almost surely, given $A_1$. 
\end{proposition}
\begin{proof}
\cite[Proposition 3.7]{msig1}.
\end{proof}
\begin{comment}
\begin{lemma}\label{lemma::localsetconditionalexpectations}
Let $h$ be a $\GFF$ on $D$ with mean $\mathcal{C}$, and fix a point $z\in D$. Suppose that $A$ is a local set for $h$ such that $\CR(z,D\setminus A)$ is almost surely positive and constant. Then $\mathcal{C}_A(z)$ is a Gaussian with mean $\mathcal{C}(z)$ and variance $\log \CR(z,D)-\log \CR(z,D\setminus A)$.
\end{lemma}

\begin{proof}

The assumption on $\CR(z,D\setminus A)$ means that we can find, by the Koebe 1/4 theorem, an $\epsilon>0$ such that $B(z,2\epsilon)\subset D\setminus A$. Therefore, by our previous arguments, we have that $h_\epsilon(z)-\mathcal{C}(z)$ is a Gaussian with mean $0$ and variance $-\log \epsilon + \log \CR(z,D)$. Similarly, the conditional law of $h_\epsilon(z)-\mathcal{C}_A(z)$ given $(A,h|_A)$, is that of a Gaussian with mean 0 and variance $-\log \epsilon + \log \CR(z, D\setminus A)$. In fact, since $\CR(z,D\setminus A)$ is almost surely constant, it has this law without conditioning on $(A, h|_A)$, and so $h_\epsilon(z)-\mathcal{C}_A(z)$ is independent of $(A,h|_A)$. This means in particular that it is independent of $\mathcal{C}_A$, and so writing $\mathcal{C}_A=h_\epsilon(z)-(h_\epsilon(z)-\mathcal{C}_A)$ the result follows.
\end{proof}
\end{comment}
\begin{proposition}\label{localsetsprocess}
Let $h$ be a $\GFF$ and $(Z(t),t\ge 0)$ a family of closed sets such that $Z(\tau)$ is local for every $Z$-stopping time $\tau$. Suppose futher that for a fixed $z\in D$, $\CR(z,D\setminus Z(t))$ is almost surly continuous and monotonic in $t$. Then, if we reparameterise time by 
$$\log \CR(z,D\setminus Z(0))-\log \CR(z,D\setminus Z(t)), $$
the process $\mathcal{C}_{Z(t)}(z)-\mathcal{C}_{Z(0)}(z)$ has a modification which is Brownian motion until the first time that $Z(t)$ accumulates at $z$. In particular, $\mathcal{C}_{Z(t)}(z)$ has a modification which is almost surely continuous in $t$. 
\end{proposition}
\begin{proof}
This is proved in \cite[Proposition 6.5]{msig1}. Since we need the argument in the proof later, we briefly recall the proof here. 
For $s\ge 0$, set 
\[\tau(s):=\inf\{t\geq 0: \log \CR(z,D\setminus Z(0))-\log \CR(z,D\setminus Z(t))=s\}.\] 
We need only show that the increments of the process $\mathcal{C}_{Z(\tau(t))}(z)$ are independent, and stationary with Gaussian distribution. 
By \cite[Lemma 6.4]{msig1}, we know that for any $s<t$, the conditional law of 
\[\mathcal{C}_{Z(\tau(t))}(z)-\mathcal{C}_{Z(\tau(s))}(z),\] 
given $(Z(\tau(s)),h|_{Z(\tau(s))})$, is a Gaussian with mean $0$ and variance 
\[\log\CR(z, D\setminus Z(\tau(s)))-\log\CR(z, D\setminus Z(\tau(t)))=t-s.\] 
This means it must also be independent of $(Z(\tau(s)),h|_{Z(\tau(s))})$, and so of $\mathcal{C}_{Z(\tau(s))}(z)$. This completes the proof. 
\end{proof}

\subsection{$\SLE_{\kappa}(\rho)$ processes}\label{subsec::slerho}
We call a compact set $K\subset\HH$ an $\HH$-hull if $H:=
 \HH\setminus K$ is simply connected. For any such hull one can show that there exists a unique conformal map $\phi$ from $H \to \HH$ which is normalized at $\infty$ in the sense that 
 \[ \phi(z)=z+\frac{2a}{z} + o(\frac{1}{z}), \quad \text{as} \; z\to \infty,\] for some constant $a$ which we call the \emph{half-plane capacity} of $K$. 
For a continuous real-valued function $(W_t, t\geq 0)$ with $W_0=0$ we can define the solution $g_t(z)$ to the \emph{chordal Loewner equation} 
 \[\partial_t g_t(z) = \frac{2}{g_t(z)-W_t}, \quad g_0(z)=z.\]
 This is well defined for each $z \in \HH$ until the first time, $\tau(z)$, that $g_t(z)-W_t$ hits 0. Setting $K_t = \{ z\in \overline{\HH}: \tau(z)\leq t\}$ and $H_t=\HH\setminus K_t$ we find that $g_t$ is the conformal map from $H_t$ to $\HH$ normalized at $\infty$, and the half-plane capacity of $K_t$ is equal to $2t$. 
 We call the family $(K_t, t\geq 0)$ the \emph{Loewner chain driven by} $(W_t, t\geq 0)$. One class of Loewner chains that we will be particularly interested in are those generated by continuous curves; that is, those for which there exists a continuous curve $\gamma$ such that $K_t$ is the hull generated by $\gamma[0,t]$ for all $t$. 
\begin{comment}
 In the sequel, the following general results regarding chordal Loewner chains will be useful.
 
 \begin{lemma}
 Suppose that $T\in (0,\infty]$ and $\gamma:[0,T)\to \bar{H}$ is a continuous, non-crossing curve with $\gamma(0)=0$. Assume that 
 \begin{itemize}
 \item $\gamma(t,T)$ is contained in the closure of the unbounded connected component of $\HH\setminus \gamma(0,t)$ and 
 \item $\gamma^{-1}(\gamma[0,t]\cup \R)$ has empty interior in $(t,T)$. 
 \end{itemize}
 Then the hull generated by $\gamma$ is a Loewner chain with almost surely continuous driving function. 
 \end{lemma}
 \begin{proof}
 \cite[Proposition 4.3]{lawler} and \cite[Proposition 6.12]{msig1}. 
 \end{proof}
 
 \begin{lemma}\label{lemma::cc_lebesguemeasurezero}
 Suppose that $\gamma$ is a continuous curve in $\bar{\HH}$ from $0$ to $\infty$ whose Loewner driving function is continuous. Then the set $\{t:\gamma(t)\in \R\}$ has Lebesgue measure $0$. 
 \end{lemma}
 \begin{proof}
 \cite[Lemma 2.5]{msig1}
 \end{proof}
\end{comment}

Chordal $\SLE_\kappa$ is the Loewner chain driven by $W_t=\sqrt{\kappa} B_t$, where $B_t$ is a standard one-dimensional Brownian motion. It is characterised by the special properties of conformal invariance and the domain Markov property. Specifically, $(\mu^{-1} K_{\mu^2t}, t\geq 0)$ has the same law as $(K_t, t\geq 0)$ for any $\mu>0$, and for any stopping time $\tau$, the law of $(f_{\tau}(K_{t+\tau}), t\geq 0)$ is the same as that of $K$. Here $f_\tau:=g_\tau - W_\tau$. 

It is known that $\SLE_\kappa$ is almost surely generated by a continuous curve for all $\kappa$. In the special case $\kappa\in [0,4]$, it has also been shown that the curve is almost surely simple. Moreover we know that $\lim_{t\to \infty} \gamma(t)=\infty$ almost surely; a property we refer to as \emph{transience}. These facts were all proved in \cite{rs}.

\begin{definition}
\label{defn::slekapparho}
Let $\rho^L$ and $\rho^R$ be finite Radon measures on $\R_-=(-\infty,0]$ and $\R_+=[0,\infty)$ respectively, and $(B_t,t\ge 0)$ be a standard one-dimensional Brownian motion. We say that $\left(W_t, (V_t^L(x))_{x\in\R_-}, (V_t^R(x))_{x\in \R_+}\right)_{t\geq 0}$ describe an $\SLE_\kappa(\rho^L;\rho^R)$ process, if they are adapted to the filtration of $B$ and the following hold: 
\begin{enumerate}
\item [(1)] The processes $W_t$, $B_t$, $\left(V_t^{L}(x)\right)_{x\in \R_-}$ and $\left(V_t^{R}(x)\right)_{x\in\R^+}$ satisfy the following SDE on time intervals where $W_t$ does not collide with any of the $V_t^{L,R}(x)$:
\begin{equation}
\label{wtsde}
dW_t  =  \sqrt{\kappa} dB_t + \left( \int_{\R_-} \frac{\rho^L(dx)}{W_t-V_t^L(x)}  \right) dt +   \left( \int_{\R_+} \frac{\rho^R(dx)}{W_t-V_t^R(x)}   \right) dt
\end{equation}
and
\begin{equation}
\label{vtxsde}
dV_t^L(x) =  \frac{2 dt}{V_t^L(x)-W_t}, \quad x\in\R_-;\quad 
dV_t^R(x) = \frac{2 dt}{V_t^R(x)-W_t}, \quad x\in\R_+.
\end{equation}
\item [(2)] We have instantaneous reflection of $W_t$ off the $V_t^{L,R}(x)$, ie. it is almost surely the case that for Lebesgue almost all times $t$ we have that $W_t\ne V_t^{L,R}(x)$ for each $x\in \R$.
%\item [(3)] We have almost surely that
%\begin{equation}\label{eqn::forcepoints_nopush}
%V_t^L(x)= x + \int_0^t \frac{2 ds}{V_s^L(x) - W_s},\quad x\in\R_-;\quad 
%V_t^R(x)= x + \int_0^t \frac{2 ds}{V_s^R(x) - W_s},\quad x\in\R_+;
%\end{equation}
%\begin{equation}\label{eqn::driving_nopush}
%W_t=\sqrt{\kappa}B_t+\int_0^t ds \int_{\R_-} \frac{\rho^L(dx)}{W_s-V_s^L(x)}+\int_0^tds \int_{\R_+} \frac{\rho^R(dx)}{W_s-V_s^R(x)}
%\end{equation}
\end{enumerate}
The $\SLE_\kappa(\rho^L;\rho^R)$ process is then defined to be the Loewner chain driven by $W$.
\end{definition}

\begin{remark}
Note that it is not immediate from the definition that such a process exists. Indeed, we will only show the existence for $\kappa=4$ and a specific subset of $(\rho^L;\rho^R)$.
\end{remark} 

We define the \emph{continuation threshold} of the process to be the be the infinum of values of $t$ for which 
\[ \text{either}\quad \rho^L\left( \{x\in \R_-: V_t^L(x)=W_t\} \right) \leq -2 \quad \text{or}\quad  \rho^R \left(\{x\in \R^+: V_t^R(x)=W_t\} \right) \leq -2.\]

Observe that the case $\rho^L \equiv 0, \rho^R \equiv 0$ corresponds simply to $\SLE_\kappa$. Another special case is when the Radon measures are purely atomic. If this occurs we instead consider $(\rho^L; \rho^R)$ to be a pair of vectors  \[\underline{\rho}^L=(\rho^L_l, \cdots, \rho^L_1),\quad \underline{\rho}^R=(\rho_1^R, \cdots, \rho_r^R)\]
with associated force points 
\[\underline{x}^L=(x^L_l<\cdots<x^L_1\leq 0),\quad \underline{x}^R=(0\leq x^R_1<\cdots x^R_r) \]
in the obvious way. In this case, it is proved in \cite[Theorem 2.2]{msig1} that a slightly stronger version of Definition \ref{defn::slekapparho} determines a unique law on $\SLE_\kappa(\underline{\rho}^L; \underline{\rho}^R)$ processes, defined for all time up until the continuation threshold. The additional condition they impose is that $W_t$, $B_t$, $\left(V_t^{L}(x)\right)_{x\in \R_-}$ and $\left(V_t^{R}(x)\right)_{x\in\R^+}$ in fact must satisfy (\ref{wtsde}) and (\ref{vtxsde}) at all times. This ensures the uniqueness in law of these processes.

%One fact we will use about these processes is that they are scale invariant: if $K$ is an $\SLE_\kappa(\underline{\rho}^L; \underline{\rho}^R)$ process and $\mu>0$, then $(\mu^{-1}K_{\mu^2t})_{t\geq 0}$ has the law of an $\SLE_\kappa(\underline{\rho}^L; \underline{\rho}^R)$ process, with force points at $\underline{x}^L/\mu, \underline{x}^R/\mu$.

Through their connection with the GFF, which we will discuss in the next section, it was shown in \cite{msig1} that $\SLE_\kappa(\underline{\rho}^L; \underline{\rho}^R)$ processes are almost surely generated by continuous curves up to and including the continuation threshold. Moreover, on the event that the continuation threshold is not hit before the curves reach $\infty$, the curves are almost surely transient. One can also show that the curves are absolutely continuous with respect to $\SLE_\kappa$  as long as they are away from the boundary. 

\subsection{Level lines of the GFF with piecewise constant boundary data}\label{subsec::levelline_piecewiseconstant}
As discussed in the introduction, the theory of level lines and flow lines  of a GFF with piecewise constant boundary data has been studied previously in a number of works, including \cite{dub}, \cite{msig1},\cite{ss} and \cite{wwll1}. We collect in this section some results that will be useful in our article.  

Suppose that $F$ is a bounded harmonic function in $\HH$ whose boundary value is piecewise constant on $\R$ and changes only finitely many times. Then $F$ can be described almost everywhere in terms of a pair of purely atomic finite Radon measures $(\rho^L; \rho^R)$, corresponding to vectors $(\underline{\rho}^L; \underline{\rho}^R)$, via the relation 
\begin{equation} \label{eqn::frho} 
F(x)=\lambda(1+\rho^R([0,x])),\quad  x\ge 0;\quad 
F(x)=-\lambda(1+\rho^L((x,0])), \quad x<0.
\end{equation}
When $\kappa=4$, which corresponds to level lines of the GFF, the following results are known for any $(\underline{\rho}^L;\underline{\rho}^R)$: (see \cite[Theorems 1.1.1 and 1.1.2]{wwll1})
\begin{itemize}
\item There exists a coupling $(K,h)$ where $K$ is an $\SLE_4(\underline{\rho}^L; \underline{\rho}^R)$ process and $h$ is a zero boundary GFF, such that $K$ is a level line of $h+F$.
\item If $h$ is a zero boundary GFF and $K$ an $\SLE_4(\underline{\rho}^L; \underline{\rho}^R)$ process, coupled such that $K$ is a level line of $h+F$, then $K$ is almost surely determined by $h$.
\end{itemize}

This allows us, for any such $F$ and an instance of the zero boundary GFF $h$ in $\HH$, to define \emph{the} level line, $\gamma$, of $h+F$. It has been shown in \cite[Theorem 1.1.3]{wwll1} that $\gamma$ is in fact almost surely continuous up to and including the continuation threshold, and it is transient when the continuation threshold is not hit. 

More generally, for any simply connected domain $D$ and $x,y$ in $\partial D$, we say that $\gamma$ is the level line of a GFF $h$ in $D$ started at $x$ and targeted at $y$, if $\phi(\gamma)$ is the level line of $h\circ \phi^{-1}$, where $\phi$ is any conformal map from $D$ to $\HH$ which sends $x$ to 0 and $y$ to $\infty$. 

One nice property of the level lines is what we call \emph{monotonicity}. Suppose that $h$ is a $\GFF$ with piecewise constant boundary values, changing only finitely many times. For $u\in\R$, we define the level line of $h$ with height $u$ to be the level line of $h+u$, and denote it by $\gamma_u$. Then, for any $u_1\ge u_2$, the level line $\gamma_{u_1}$ lies to the left of $\gamma_{u_2}$ almost surely, see \cite[Theorem 1.1.4]{wwll1}.

Another property of the level lines is their \emph{reversibility}. Suppose that $h$ is a $\GFF$ with piecewise constant boundary values changing only finitely many times. Let $\gamma$ be the level line of $h$ from $0$ to $\infty$ and $\gamma'$ be the level line of $-h$ from $\infty$ to $0$. Then, on the event that neither hit their continuation thresholds before reaching their target points, we have $\gamma=\gamma'$ almost surely as sets. This implies the \emph{reversibility} of the $\SLE_4(\underline{\rho}^R; \underline{\rho}^R)$ process: conditioned on the event that the continuation threshold is not hit, the time reversal of the process is another $\SLE_4(\underline{\rho}^L; \underline{\rho}^R)$ process, now from $\infty$ to $0$ in $\HH$ with appropriate weights and force points, conditioned not to hit its continuation threshold. See \cite[Theorem 1.1.6]{wwll1}.
\begin{comment}
Finally, we remark that the level lines satisfy a \emph{target independence property}. If $y_2<x<y_1$ on $\R$ and $h+F$ is a Gaussian free field in $\HH$ with some piecewise constant boundary data, let $\gamma_i$, for $i=1,2$ be the level line of $h+F$ started from $x$ and targeted at $y_i$. Let $T_i$ be the first time that $y_1$ and $y_2$ are disconnected by $\gamma_i$. Then, almost surely, the paths $\gamma_1$ and $\gamma_2$ coincide up to and including their first disconnecting time (modulo a time change). After this time, and given $(\gamma_1[0,T_1], \gamma_2[0,T_2])$, the two paths continue to evolve independently towards their target points.
\end{comment}
Finally, we include a list of results from \cite{wwll1} that will be useful for the later proofs.
\begin{lemma}\label{lem::results_piecewiseconstant}
Suppose that $h$ is a zero-boundary $\GFF$ and $F$ is the bounded harmonic extension of the piecewise constant boundary data which changes finitely many times. Let $\gamma$ be the level line of $h+F$. We already know that $\gamma$ is almost surely continuous up to and including the continuation threshold.  
\begin{enumerate}
\item [(1)] \cite[Theorem 1.1.3]{wwll1} The curve $\gamma$ is almost surely simple and is continuous up to and including the continuation threshold.
\item [(2)] \cite[Remark 2.5.15]{wwll1} For any open interval $I$ of $(-\infty, 0)\cup (0,\infty)$, assume that 
\[\text{either } F(x)\ge \lambda,\quad \forall x\in I,\quad \text{or } F(x)\le -\lambda, \quad \forall x\in I.\]
Then almost surely $\gamma\cap I=\emptyset$.
\item [(3)] \cite[Proposition 2.5.11]{wwll1} For any point $x_0\in (0,\infty)$, assume that there exists $c>0$ such that $F\ge -\lambda+c$ in a neighborhood of $\{x_0\}$, then almost surely $\gamma$ does not hit $\{x_0\}$. Symmetrically, for $x_0\in (-\infty, 0)$, assume that there exists $c>0$ such that $F\le \lambda-c$ in a neighborhood of $\{x_0\}$, then almost surely $\gamma$ does not hit $\{x_0\}$.
\end{enumerate}
\end{lemma}

\subsection{First generalizations to the GFF with general boundary data}
In this section, we generalize some results concerning level lines with piecewise constant boundary data to general boundary data. In fact, the ideas in the proof for Lemma \ref{lem::coupling_iff_martingale} when the boundary condition is piecewise constant \cite[Lemmas 2.4-2.6]{ss} work for general boundary data with proper adjustment. In order to be self-contained, we still give a complete proof here.

\begin{lemma}
\label{lem::cr_evolution}
Suppose that $(K_t, t\ge 0)$ is a Loewner chain driven by a continuous process $(W_t, t\ge 0)$. Denote by $(g_t, t\ge 0)$ the corresponding sequence of conformal maps and $f_t=g_t-W_t$ the centered conformal maps. For any fixed $z\in \HH$, define
\[C_t(z)=\log\CR(z,\HH)-\log\CR(z,\HH\setminus K_t).\]
Then, we have that 
\[dC_t(z)=\frac{4\Im(f_t(z))^2}{|f_t(z)|^4} dt.\]
\end{lemma}

\begin{proof}
The conformal radius $\CR(z, \HH\setminus K_t)$ is equal to $2/|\phi_t'(z)|$ for $\phi_t$ any conformal map from $\HH \setminus K_t$ to $\mathbb{H}$ which sends $z$ to $i$, an example of which is given by $m_t \circ f_t$, where $m_t: \HH \to \HH$ is the M\"{o}bius transformation defined by 
$$m_t(w)= \frac{\Im(f_t(z))w}{\Re(f_t(z))^2 + \Im(f_t(z))^2 -\Re(f_t(z))w}.$$ This gives us that 
\[C_t(z)-C_0(z)=-\log 2 + \Re(\log m_t'(f_t(z))) + \Re(\log g_t'(z)).\] However, in this case we can calculate $m_t'(f_t(z))$ explicitly, and find that $-\Re(\log m_t'(f_t(z)))=\log \Im f_t(z).$ Since we also know that 
\[dg_t'(z)=\frac{-2g_t'(z)}{f_t(z)^2}dt, \quad d\Im(f_t(z))=\frac{-2\Im(f_t(z))}{|f_t(z)|^4} dt,\] we can compute 
\[dC_t(z)=\frac{4\Im(f_t(z))^2}{|f_t(z)|^4} dt,\] and this implies the result.
\end{proof}

\begin{lemma}\label{lem::coupling_iff_martingale}
Assume the same notations as in Definition \ref{def::gff_levelline}. 
Suppose that the Loewner chain $K$ is almost surely generated by a random continuous curve $\gamma$ on $\overline{\HH}$ from 0 to $\infty$ whose driving function $W$ is almost surely continuous. For $z\in\HH$ and $t\ge 0$, set
\[\tau(t)=\inf\{s: \log\CR(z, \HH)-\log\CR(z, \HH\setminus K_s)=t\}.\]
Then the pair $(h, K)$ can be coupled as in Definition \ref{def::gff_levelline} if and only if $(\eta_{\tau(t)}(z), t\ge 0)$ is a Brownian motion with respect to the filtration generated by $(W_{\tau(t)}, t\geq 0)$ for any $z\in\HH$.
\end{lemma}
\begin{comment}
{\color{blue} Note, we could also go back to the previous version of this Lemma (which we still have in comments) but I'd really prefer to use something stronger like the above if possible. I think we will have to use the same arguments anyway and this will mean that we don't have to change too much in Section 6. Also, we need to take care not to affect the proof of Lemma \ref{lemma::nothitting}. If you agree that something like the above version of the Lemma works, then I think Lemma \ref{lemma::nothitting} will not be affected, and I've made an adapted version of Section 6 which should work. I've also kept the old one too, with only the other corrections in, so you can choose which to compile. See below.}
\end{comment}
\begin{proof}
If $(h,K)$ is coupled as in Definition \ref{def::gff_levelline}, by  Proposition \ref{localsetsprocess}, we know that $(\eta_{\tau(t)}(z), t\ge 0)$ is a Brownian motion. Moreover, by the proof of Proposition \ref{localsetsprocess}, we see that, for $s<t$, the variable $\eta_{\tau(t)}(z)-\eta_{\tau(s)}(z)$ is independent of $K_{\tau(s)}$ and has the law of Gaussian with mean zero and variance $t-s$. This implies that $(\eta_{\tau(t)}(z), t\ge 0)$ is a Brownian motion with respect to the filtration generated by $(W_{\tau(t)}, t\ge 0)$.

For the converse, assume that, for each $z\in \HH$, the process $(\eta_{\tau(t)}(z), t\ge 0)$ is a Brownian motion with respect the filtration generated by $(W_{\tau(t)}, t\ge 0)$. We will begin by showing that there exists a Brownian motion $(B_t, t\ge 0)$ (with respect to the filtration of $(W_t, t\ge 0)$) such that, for all $z$, we have 
\begin{equation}\label{eqn::eta_z_bm}
d\eta_t(z)=\Im{\frac{2}{f_t(z)}}dB_t.
\end{equation}
Define 
\begin{equation}\label{eqn::eta_z_bv}
U_t(z)=\eta_t(z)+\arg(f_t(z)).
\end{equation}
We have the following observations.
\begin{itemize}
\item By the definition of $\eta_t(\cdot)$, we know that $U_t(\cdot)$ is the bounded harmonic function on $\HH\setminus K_t$ with the boundary values given by $F+2\lambda$ on $\R_-\setminus K_t$, $\lambda$ along the boundary of $K_t$, and $F$ on $\R_+\setminus K_t$. 
Therefore, for fixed $z$, process $(U_t(z), t\ge 0)$ is of bounded variation and is measurable with respect to the filtration generated by $(W_t, t\ge 0)$. 
\item By the assumption, for fixed $z$, the process $(\eta_t(z), t\ge 0)$ is a Brownian motion when parameterized by 
\[C_t(z)=\log\CR(z, \HH)-\log\CR(z, \HH\setminus K_t).\] Thus, by Lemma \ref{lem::cr_evolution}, we see that 
\[d\langle \eta_t(z)\rangle=dC_t(z)=4\left(\Im{\frac{1}{f_t(z)}}\right)^2 dt.\]
Moreover, since $(\eta_{\tau(t)}(z), t\ge 0)$ is a Brownian motion with respect to the filtration of $(W_{\tau(t)}, t\ge 0)$, we know that, for any $s<t$, the variable $\eta_t(z)-\eta_s(z)$ is independent of $(W_u, u\le s)$ (we implicitly use the fact that $K$ is a Loewner chain generated by a continuous curve with continuous driving function), thus the process $(\eta_t(z), t\ge 0)$ is a local martingale with respect to the filtration of $(W_t, t\ge 0)$.
\end{itemize}
Combining these two facts, we know that $\arg(f_t(z))=\arg(g_t(z)-W_t)$ is a semimartingale, and hence $W_t$ is a semimartingale at least up to the first time that $z$ is swallowed by $K_t$. 
Note that 
\[d \arg(f_t(z))=\Im{\frac{-1}{f_t(z)}}dW_t+\Im{\left(\frac{2dt}{(f_t(z))^2}-\frac{d\langle W_t\rangle}{2( f_t(z))^2}\right)},\] 
and therefore, we have $d\langle W_t\rangle=4dt$ for all such $t$. Note however that the process $W$ does not depend on $z$, and since we can always choose $z$ far away as we want, we can argue that $(W_u, 0\le u\le t)$ is a semimartingale up to time $t$ for any $t>0$, with $d\langle W_t\rangle=4dt$. 
Thus, there exists a Brownian motion $(B_t, t\ge 0)$ and a process of bounded variation $(V_t, t\ge 0)$ such that $W_t= 2B_t-V_t$.
We emphasize that the processes $B$ and $V$ do not depend on $z$. Plugging in Equation (\ref{eqn::eta_z_bv}), we have that 
\begin{equation}\label{eqn::W_bv}
d\eta_t(z)=\Im{\frac{2}{f_t(z)}}dB_t,\quad dU_t(z)=\Im{\frac{1}{f_t(z)}}dV_t,
\end{equation}
as desired.

With equation (\ref{eqn::eta_z_bm}) in hand, we know that for  $z,w\in\HH$, we have 
\[d\langle \eta_t(z),\eta_t(w) \rangle = \Im \left(\frac{1}{f_t(z)}\right) \Im \left(\frac{1}{f_t(w)}\right)dt .\]
We know that, for each $z$, $\eta_t(z)$ is a continuous martingale. We can also extend the definition of $\eta_t(z)$ by setting it equal to its limit as $s\uparrow \tau(z)$ at all times after $\tau(z)$. We further define for $z,w\in \HH$ and $t\leq \tau(z)\wedge \tau(w)$, 
\[G_t(z,w):=G_\HH(f_t(z),f_t(w))\]
where we again extend this to all times after $\tau(z)\wedge \tau(w)$, by setting it constant and equal to its limit as $t\uparrow \tau(z) \wedge \tau(w)$. Observe that, in each connected component of $\HH\setminus \gamma[0,t]$, the function $\eta_t$ is the bounded harmonic function with boundary values shown in Figure \ref{fig::etaboundaryvalues}.
\begin{figure}[ht!]
\centering
\includegraphics[width=0.35\textwidth]{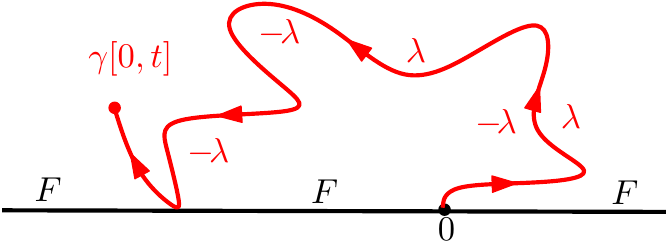}
\caption{The function $\eta_t(\cdot)$ is a harmonic function in each connected component of $\HH\setminus \gamma[0,t]$ with boundary values as above.}
\label{fig::etaboundaryvalues}
\end{figure} 
We also know that $G_t(z,w)$ is non-decreasing in $t$ for any fixed $z,w$. 

Putting all of the above together, we can deduce by stochastic calculus that for any $p\in H_s(\HH)$, 
$(\eta_t, p)$ is a continuous martingale with 
\[d\langle (\eta_t,p)\rangle = -dE_t(p),\quad \text{where}\quad E_t(p):= \int \int p(z)p(w)G_t(z,w) \, d^2z d^2w.\]

Now we are ready to show that the pair $(h,K)$ is coupled as in Definition \ref{def::gff_levelline}. Since for each $z,w\in \HH$ and non-negative $p\in H_s(\HH)$ we have that $\eta_t(z)$ is a martingale and $G_t(z,w)$, $E_t(p)$ are non-decreasing, it must be that all the limits $\eta_\infty(z)$, $G_\infty(z,w)$ and $E_\infty(p)$ exist. We let $\tilde{h}$ be equal to $\eta_\infty - \eta_0$ plus a sum of independent zero boundary  GFF's; one in each connected component of $\HH\setminus \gamma$. To show that $(K,\tilde{h})$ are coupled in the correct way we must verify that the marginal law of $\tilde{h}$ is that of a zero boundary  GFF in $\HH$, and that $(K,\tilde{h})$ satisfies the correct domain Markov property. This amounts to showing that for each non-negative $p\in H_s(\HH)$:
\begin{itemize}
\item $(\tilde{h},p)$ is a Gaussian with mean $0$ and variance $E_0(p)$.
\item For any $K$-stopping time $\tau$, the conditional law of $\left((\tilde{h}+\eta_0)|_{\HH\setminus K_\tau},p\right)$ given $K_\tau$ is a Gaussian with mean $(\eta_\tau,p)$ and variance $E_\tau(p)$.
\end{itemize}
To see the first point, for any $\mu>0$ we calculate
\begin{align*}
\E[\exp(-\mu(\tilde{h},p))] &=  \E[\E[\exp(-\mu(\tilde{h},p))|K]] \\
& = \E\left[\exp\left(-\mu(\eta_\infty-\eta_0, p)-\frac{\mu^2}{2}E_\infty(p)\right)\right] \\
& = \E\left[\exp\left(-\mu(\eta_\infty-\eta_0, p)+\frac{\mu^2}{2}(E_0(p)-E_\infty(p)\right)\right]\exp\left(-\frac{\mu^2}{2}E_0(p)\right)\\
& = \exp\left(-\frac{\mu^2}{2}E_0(p)\right),
\end{align*}
where the last line follows from the fact that $(\eta_t,p)$ is a continuous bounded martingale with mean $\eta_0(p)$ and quadratic variation $E_0(p)-E_\infty(p).$ The second point follows similarly, replacing the initial expectation with a conditional one.
\end{proof}

\section{Non-boundary intersecting regime}\label{sec::nonboundary_intersecting}
All the conclusions in Sections \ref{sec::nonboundary_intersecting} and \ref{sec::monotonicity} are proved in \cite{msig1, wwll1} for level lines with piecewise constant boundary data and constant height difference. Although many of the ideas from these papers are fundamental to our proofs, there are several places where they fail for general boundary data. Therefore, we treat the general case here and give complete proofs in the next two sections.
\begin{lemma}\label{lemma::nothitting}
 Suppose that $\gamma$ is a random continuous curve from 0 to some $\gamma$-stopping time $T$ with almost surely continuous driving function. Assume that $\gamma$ is coupled with a zero boundary $\GFF$ $h$ as a level line of $h+F$ up to time $T$ where 
\[F(x)\ge -\lambda,\quad \forall x<0;\quad F(x)\ge \lambda,\quad \forall x\ge 0.\]
Then almost surely $\gamma[0,T] \cap (0,\infty)=\emptyset$.
\end{lemma}

\begin{proof}
Assume the same notations as in Definition \ref{def::gff_levelline}. 
First, for any $z \in \HH$, define $U_t(z)$ in the same way in Equation (\ref{eqn::eta_z_bv}), and we will explain that the process $(U_t(z), 0\le t\le T)$ is non-increasing. 
By the definition of $\eta_t(z)$, we know that $U_t(\cdot)-\lambda$ is the harmonic function on $\HH\setminus K_t$ with the boundary values given by $F+\lambda\ge 0$ on $\R_-\setminus K_t$, zero along the boundary of $K_t$, and $F-\lambda\ge 0$ on $\R_+\setminus K_t$. This harmonic function is non-increasing in $t$, and thus $U_t(z)$ is non-increasing.  

Next, we will show that the process 
\[Z_t:=V_t^R(0^+)-W_t\]
cannot hit zero. 
By the proof of Lemma \ref{lem::coupling_iff_martingale}, we know that there exist a Brownian motion $B$ and a process of bounded variation $V$ such that $W=2B-V$ and Equation (\ref{eqn::W_bv}) holds. Since $U_t(z)$ in non-increasing in $t$, the process $V_t$ is non-decreasing in $t$ up to the time that $z$ is swallowed. However, the process $V$ does not depend on $z$, and since we can always choose $z$ far away, we have that $(V_u, 0\le u\le t)$ is non-decreasing in $u$ for any $t>0$. 
Then we have, for all $t$, 
\[d Z_t \ge -2dB_t+\frac{2dt}{Z_t}.\]
We can compare $Z_t/2$ with a Bessel process of dimension 2, and so may conclude that $Z_t$ cannot hit 0. This implies that the curve cannot hit $(0,\infty)$. 
\end{proof}

\begin{remark}
The proof of Lemma \ref{lemma::nothitting} also applies to the case when 
\[F(x)\le -\lambda,\quad x<0;\quad F(x)\le\lambda,\quad x\ge 0\]
by symmetry. In this case we see that for $\gamma$ satisfying the same conditions as in Lemma \ref{lemma::nothitting}, we have $\gamma[0,T]\cap(-\infty,0)=\emptyset$ almost surely. 
\end{remark}

\begin{figure}[ht!]
\begin{subfigure}[b]{0.48\textwidth}
\begin{center}
\includegraphics[width=0.73\textwidth]{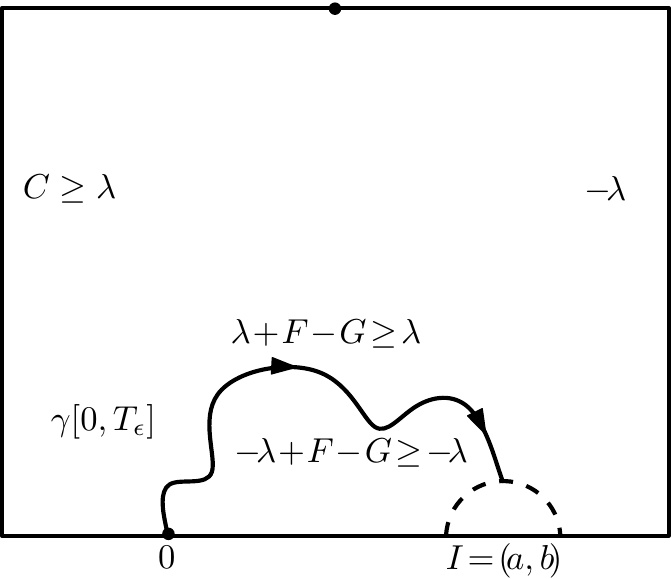}
\end{center}
\caption{The boundary value of $-h-G$ given $\gamma[0,T_{\eps}]$.}
\end{subfigure}
$\quad$
\begin{subfigure}[b]{0.48\textwidth}
\begin{center}\includegraphics[width=0.73\textwidth]{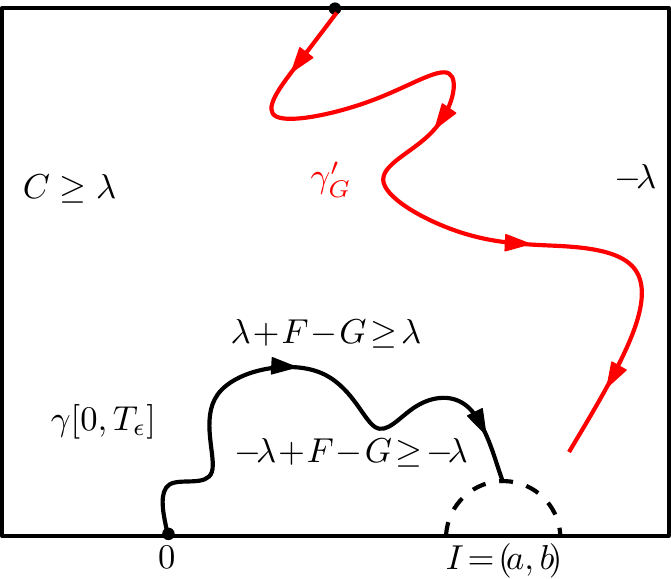}
\end{center}
\caption{$\gamma_G'$ cannot hit the left side of $\gamma[0,T_{\eps}]$ or $(-\infty, 0]$.}
\end{subfigure}
\caption{\label{fig::nonboundary_intersecting} Explanation of the boundary values in the proof of Lemma \ref{lem::nonboundary_intersecting}.}
\end{figure}

\begin{lemma}\label{lem::nonboundary_intersecting}
Suppose that $\gamma$ is a random continuous curve from 0 to some $\gamma$-stopping time $T$ with almost surely continuous driving function. Assume that $\gamma$ is coupled with a zero boundary $\GFF$ $h$ as a level line of $h+F$ up to time $T$ where
\[F(x)\le -\lambda,\quad \forall x<0;\quad F(x)\ge\lambda,\quad \forall x\ge0.\] Then, almost surely, the curve $(\gamma(t), 0\le t\le T)$ does not hit the boundary except the two end points.
\end{lemma}

\begin{proof}
It is sufficient to show that, for any $0<a<b<\infty$, the curve $\gamma$ does not hit the interval $I=(a,b)$. We prove by contradiction. 

Suppose that $\gamma$ does hit $I$ with positive probability, and on this event, define $T_{\eps}$ to be the first time that $\gamma$ gets within $\eps$ of $I$. Since $F$ is bounded, suppose that $F\ge -C$ for some finite $C\ge \lambda$. Let $G$ be the bounded harmonic extension of the function which is equal to $-C$ on $\R_-$ and is equal to $\lambda$ on $\R_+$. Note that $F\ge G$. Let $\gamma_G'$ be the level line of $-h-G$ from $\infty$ to 0. By Lemma \ref{lem::results_piecewiseconstant}(1) and (2), we know that 
$\gamma_G'$ is almost surely continuous and transient; and that $\gamma_G'$ almost surely does not hit $I$. 

Let $\tilde{h}$ be $h$ restricted to the unbounded connected component of $\HH\setminus \gamma[0,T_{\eps}]$, then conditionally on $\gamma[0,T_{\eps}]$, the field $-\tilde{h}-G$ is a $\GFF$ with boundary data as shown in Figure \ref{fig::nonboundary_intersecting}(a). Moreover, given $\gamma[0,T_{\eps}]$, the curve $\gamma_G'$ is coupled with $\tilde{h}$ so that it is a level line of $-\tilde{h}-G$ up until the first time that $\gamma_G'$ hits $\gamma[0, T_{\eps}]$ (by Propositions \ref{propn::localsetunions} to \ref{propn::morelocalsets}). Since $F-G$ is positive on $\HH$, we see from Lemma \ref{lemma::nothitting} that $\gamma_G'$ cannot hit the left side of $\gamma[0,T_{\eps}]$ or $(-\infty, 0]$ before hitting the right side of $\gamma[0,T_{\eps}]$ or the tip $\gamma(T_{\eps})$, see Figure \ref{fig::nonboundary_intersecting}(b). In any case, this implies that $\gamma_G'$ has to get within $\eps$ of $I$. Since this holds for any $\eps>0$ on the event that $\gamma$ hits $I$ and $\gamma_G'$ is continuous, we can conclude that $\gamma_G'$ hits $I$ with positive probability, contradiction.
\end{proof}

\begin{lemma}\label{lem::nonboundary_intersecting_simple}
Assume the same notations as in Lemma \ref{lem::nonboundary_intersecting}. Then $\gamma$ is almost surely simple.
\end{lemma}
\begin{proof}
First, we argue that, for any $\gamma$-stopping time $\tau$, we have $\gamma[0,\tau)\cap \gamma(\tau,T)=\emptyset$ almost surely. Given $\gamma[0,\tau]$, denote by $\tilde{h}$ the restriction of $h+F$ to $\HH\setminus \gamma[0,\tau]$ (since $\gamma$ does not hit the boundary, this set only has one connected component). 
By the domain Markov property in Definition \ref{def::gff_levelline}, we know that, given $\gamma[0,\tau]$, the curve $\gamma|_{[\tau,T)}$ is coupled with $\tilde{h}$ as its level line. Note that the boundary value of $\tilde{h}+F$ is $F\le -\lambda$ on $\R_-$, is $-\lambda$ along the left side of $\gamma[0,\tau]$, is $\lambda$ along the right side of $\gamma[0,\tau]$, and is $F\ge \lambda$ along $\R_+$. By Lemma \ref{lem::nonboundary_intersecting}, we know that $\gamma(\tau, T)$ cannot hit $\gamma(0,\tau)$.

Next, we show that $\gamma$ is almost surely simple. For any $q>0$, define $A_q$ to be the event that $\gamma(0,q)\cap\gamma(q,T)\neq\emptyset$. If $\gamma$ has double point, then $A_q$ happens for some positive rational $q$, since $\gamma$ is continuous. However, by the above argument, we know that $\cup_{q\in\QQ_+}A_q$ has zero probability. Therefore, $\gamma$ is almost surely simple.  
\end{proof}

\begin{proposition} Suppose that $h$ is a zero boundary  $\GFF$ and that $F$ is bounded and satisfies
\[F(x)\le -\lambda,\quad \forall x<0;\quad F(x)\ge\lambda,\quad \forall x\ge0.\]
Suppose that $\gamma$ (resp. $\gamma'$) is a random continuous transient curve from $0$ to $\infty$ (resp. from $\infty$ to 0) with almost surely continuous driving function.  

Assume that $\gamma$ is coupled with $h$ as a level line of $h+F$, that $\gamma'$ is coupled with $h$ as a level line of $-h-F$, and that the triple $(h,\gamma,\gamma')$ are coupled so that $\gamma$ and $\gamma'$ are conditionally independent given $h$. Then almost surely $\gamma'$ equals $\gamma$. In particular, this implies that $\gamma$ is almost surely determined by $h$.
\end{proposition}

\begin{proof}
First, we argue that, for any $\gamma'$-stopping time $\tau'$, given $\gamma'[0,\tau']$, the curve $\gamma$ almost surely first exits $\HH\setminus \gamma'[0,\tau']$ at $\gamma'(\tau')$. Denote by $\tilde{h}$ the restriction of $h$ to $\HH\setminus\gamma'[0,\tau']$. Given $\gamma'[0,\tau']$, the curve $\gamma$ is coupled with $h$ as a level line of $\tilde{h}+F$. The boundary value of $\tilde{h}+F$ is $F\le-\lambda$ on $\R_-$, is $-\lambda$ along the left side of $\gamma'[0,\tau']$, is $\lambda$ along the right side of $\gamma'[0,\tau']$, and is $F\ge\lambda$ on $\R_+$. Thus, by Lemma \ref{lem::nonboundary_intersecting}, we know that $\gamma$ must exit $\HH\setminus \gamma'[0,\tau']$ at $\gamma'(\tau')$.  

Next, we show that $\gamma$ and $\gamma'$ are equal. Since $\gamma$ hits $\gamma'[0,\tau']$ for the first time at $\gamma'(\tau')$ for any $\gamma'$-stopping time $\tau'$, we know that $\gamma$ hits a dense countable set of points along $\gamma'$ in reverse chronological order. By symmetry, $\gamma'$ hits a dense countable set of points along $\gamma$. Since both $\gamma$ and $\gamma'$ are continuous simple curves, the two curves (viewed as sets) are equal.
\end{proof}

\section{Monotonicity}\label{sec::mono}
\label{sec::monotonicity}
\begin{lemma}\label{lemma::absolutecontinuity}
Suppose that $h$ is a zero boundary  $\GFF$ and that $F$ is bounded. Suppose that $\gamma$ is a random continuous curve from 0 to some $\gamma$-stopping time $T$
 with almost surely continuous driving function. Assume that $\gamma$ is coupled with $h$ as a level line of $h+F$ up to time $T$. 
 \begin{enumerate}
 \item [(1)] Then the curve $(\gamma(t), 0\le t\le T)$ almost surely does not intersect any open interval $I$ of $(0,\infty)$ such that 
\[ F(x)\ge \lambda \quad \forall x\in I.\]
Symmetrically, it does not intersect any open interval of $(-\infty,0)$ where $F(x)\le -\lambda$.
\item [(2)] In addition, if $(\gamma(t), 0\le t\le T)$ is almost surely simple, then it does not hit any open interval $I$ of $(-\infty,0)$ where $F(x)\ge \lambda$.
Symmetrically, it does not intersect any open interval of $(0,\infty)$ where $F(x)\le -\lambda$. 
 \end{enumerate} 
\end{lemma}

\begin{proof} [Proof of Lemma \ref{lemma::absolutecontinuity}, Item (1)]
We first show the conclusion when $I=(a,b)$ for $0<a<b$ and $F(x)\geq \lambda, \forall x\in I$. Pick $\tilde{a},\tilde{b}$ such that $a<\tilde{a}<\tilde{b}<b$. It is sufficient to show that, for any such $\tilde{a}, \tilde{b}$, the curve $(\gamma(t), 0\le t\le T)$ does not hit the interval $\tilde{I}=[\tilde{a}, \tilde{b}]$. 
We prove by contradiction. Suppose that the curve $(\gamma(t), 0\le t\le T)$ hits $\tilde{I}$ with positive probability.
Since $F$ is bounded, we have that $F\ge -C$ for some $C\ge \lambda$. Let $G$ be the bounded harmonic extension of the function which is equal to $-C$ on $\R_-\cup(0,a)\cup(b,\infty)$ and $\lambda$ on $(a,b)$. Note that $F\ge G$. Let $\gamma_G'$ be the level line of $-h-G$ from $b$ to $a$. Note that since $G$ is piecewise constant we know by Lemma \ref{lem::results_piecewiseconstant}(1) that the curve $\gamma_G'$ is continuous from $b$ to $a$, and the boundary data also means, by Lemma \ref{lem::results_piecewiseconstant}(2), that it does not hit $\tilde{I}$. 
This means we can repeat the same argument as in the proof of Lemma \ref{lem::nonboundary_intersecting} to show that $\gamma_G'$ hits $\tilde{I}$ with positive probability and obtain a contradiction. 
\end{proof}

\begin{proof}[Proof of Lemma \ref{lemma::absolutecontinuity}, Item (2)]
Now, let $I=(a,b)$ for $a<b<0$, and suppose that $F(x)\ge\lambda, \forall x\in I$ and $(\gamma(t), 0\le t\le T)$ is almost surely simple. It will be sufficient for us to prove that $\gamma$ does not hit $\tilde{I}=[\tilde{a},\tilde{b}]$ for any $a<\tilde{a}<\tilde{b}<b$. First note that if $\gamma$ hits $[-\infty,a]$ before hitting $I$, since $\gamma$ grows towards $\infty$, it can never hit $\tilde{I}$ thereafter and we are done. If not, let $\varphi$ be the M\"{o}bius transform of $\HH$ that sends the triplet $(b, 0, \infty)$ to $(\infty, 0, 1)$. Then $(\varphi(\gamma(t)), 0\le t\le T)$ is a continuous curve, coupled with a zero-boundary GFF $\tilde{h}$ as a level line of $\tilde{h}+F\circ\varphi^{-1}$ until the first time it hits $[1,\infty]$. By Item (1), we know that $(\varphi(\gamma(t)), 0\le t\le T)$ cannot hit the interval $(\varphi(a),\infty)$ before this time. Thus $(\gamma(t), 0\le t\le T)$ cannot hit $\tilde{I}$ without first hitting the point $b$.
Let $\tau$ be the time at which $\gamma$ hits $b$, setting $\tau=T$ if this never happens, and $\tau'$ be the first time at which $\gamma$ hits $\tilde{I}$, again setting $\tau'=T$ if necessary. By the previous reasoning, if we do not have $\{\tau<\tau'<T\}$ then we are done, so assume this occurs with positive probability. On this event, since $\gamma$ is a continuous curve with continuous Loewner driving function, we see that $\{\tau\leq t \leq \tau': \gamma(t)\in \R\}$ has Lebesgue measure 0, and so there exists a time $\tau<\sigma<\tau'$ with $\gamma(\sigma)\notin \R$. Let $w$ be the left-most point in $(\gamma(t),0\leq t\leq \sigma)\cap (-\infty,0)$. Then applying the same argument as above, now to $(\gamma(t),\sigma\leq t \leq T)$ in the domain $\HH \setminus (\gamma(t),0\leq t\leq \sigma)$ with $(a,b)$ replaced by $(a,w)$, we see that $(\gamma(t),\sigma\leq t \leq T)$ must first hit $w$ before it can hit $\tilde{I}$. This is a contradiction to the simplicity of $\gamma$. 
\end{proof}

\begin{lemma}\label{lemma::nothitendpoint}
Suppose that $h$ is a zero boundary  $\GFF$ and that $F$ is bounded. Suppose that $\gamma$ is a random continuous curve from 0 to some $\gamma$-stopping time $T$ with almost surely continuous driving function. Assume that $\gamma$ is coupled with $h$ as a level line of $h+F$ up to time $T$. 
\begin{enumerate}
\item [(1)] For any fixed point $x_0\in (0,\infty)$, if  there exists $c>0$ such that $F\ge -\lambda+c$ in a neighborhood of $x_0$, then 
the curve $(\gamma(t), 0\le t\le T)$ almost surely does not hit $\{x_0\}$.  Symmetrically, for any fixed point $x_0\in (-\infty,0)$, if  there exists $c>0$ such that $F\le \lambda-c$ in a neighborhood of $x_0$, then 
the curve $(\gamma(t), 0\le t\le T)$ almost surely does not hit $\{x_0\}$. 
\item [(2)] If there exists $X\in (0,\infty)$ and $c>0$ such that 
\[F(x)\ge\lambda,\quad \forall x\in (-X,\infty);\quad F(x)\ge -\lambda+c,\quad \forall x\in (X,\infty)\]
 and in addition the curve $(\gamma(t), 0 \leq t \leq T)$ is almost surely simple,
then $(\gamma(t), 0\le t\le T)$ almost surely does not hit $\infty$. Symmetrically, if there exists $X\in(0,\infty)$ and $c>0$ such that \[F(x)\le -\lambda,\quad \forall x\in (X,\infty);\quad F(x)\le \lambda-c,\quad \forall x\in (-X,\infty)\] and the curve $(\gamma(t), 0 \leq t \leq T)$ is almost surely simple, then $(\gamma(t), 0\le t\le T)$ almost surely does not hit $\infty$
\end{enumerate}
\end{lemma}
We point out that Item (2) is not a consequence of Item (1) in Lemma \ref{lemma::nothitendpoint}. In fact, if $F$ is piecewise constant and $F(x)\in (-\lambda+c, \lambda-c)$ on $(-\infty,-X)\cup(X,\infty)$ for $c\in (0,\lambda)$, then the level line of $h+F$ is transient, and hence hits $\infty$ almost surely.
\begin{proof}[Proof of Lemma \ref{lemma::nothitendpoint}, Item (1)]
We may assume that $F\ge-\lambda+c$ on $(a,b)$ where $0<a<x_0<b$ and again prove by contradiction. 
Suppose that the curve $(\gamma(t), 0\le t\le T)$ does hit $\{x_0\}$ with some positive probability.
Since $F$ is bounded, suppose that $F\ge -C$ for some $C\ge \lambda$. Let $G$ be the function which is equal to $-C$ on $\R_-\cup (0,a)\cup (b,\infty)$, and is  $(-\lambda+c)\wedge\lambda$ on $(a,b)$. Note that $F\ge G$. Let $\gamma_G'$ be the level line of $-h-G$ from $b$ to $a$. Note that $\gamma_G'$ is continuous and does not the point $\{x_0\}$ by Lemma \ref{lem::results_piecewiseconstant}(3). Thus we can repeat the same argument as in the proof of Lemma \ref{lem::nonboundary_intersecting} and show that $\gamma_G'$ hits $\{x_0\}$ with positive probability, which is a contradiction.
\end{proof}

\begin{figure}[ht!]
\begin{subfigure}[b]{0.48\textwidth}
\begin{center}
\includegraphics[width=0.73\textwidth]{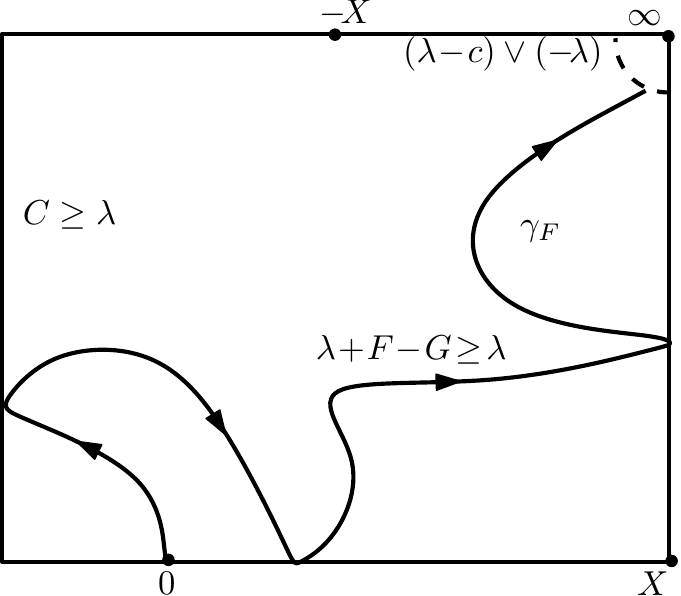}
\end{center}
\caption{Suppose that $\gamma_F$ hits $\infty$ with positive probability. Let $T_{\eps}$ be the first time that it enters $\{z: |z|>1/\eps\}.$
Since $F\ge \lambda$ on $(-\infty,-X)$, the curve $\gamma_F$ can never hit $(-\infty,-X)$. Given $\gamma_F[0,T_{\eps}]$, the boundary data of the field $-h-G$ is shown in this figure.}
\end{subfigure}
$\quad$
\begin{subfigure}[b]{0.48\textwidth}
\begin{center}\includegraphics[width=0.73\textwidth]{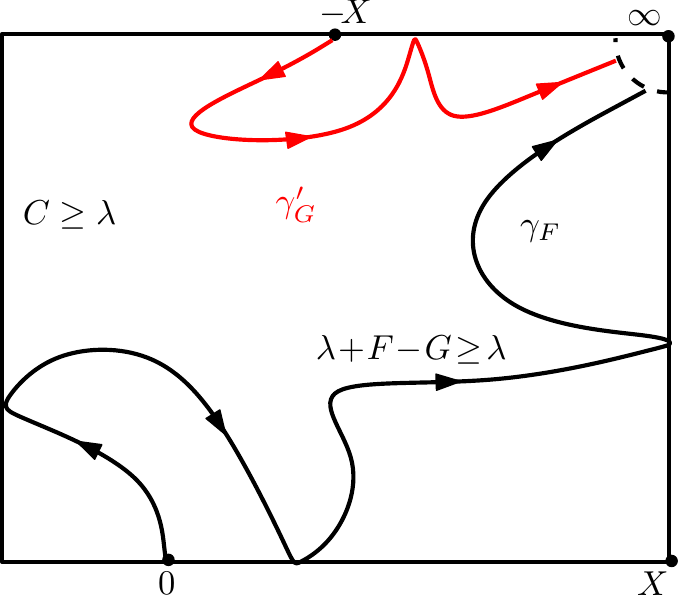}
\end{center}
\caption{By the choice of $G$, we see that $\gamma_G'$ cannot hit the union of $(-X,0)$ and the left side of $\gamma_F[0, T_{\eps}]$ before hitting $\gamma_F[0, T_{\eps}]$. Therefore, $\gamma_G'$ has to enter $\{z: |z|>1/\eps\}.$ This holds for all $\eps>0$, thus $\gamma_G'$ hits $\infty$ with positive chance, contradiction.}
\end{subfigure}
\caption{\label{fig::nohitting_target} Explanation of the boundary values in the proof of Lemma \ref{lemma::nothitendpoint}, Item (2).}
\end{figure}

\begin{proof}[Proof of Lemma \ref{lemma::nothitendpoint}, Item (2)]
We prove by contradiction. Suppose that $\gamma$ does hit $\infty$ with positive probability. 
Since $F$ is bounded, suppose that $F\ge -C$ for some $C\ge\lambda$. Let $G$ be the function which is equal to $-C$ on $(-X,0)\cup(0,X)$ and is $(-\lambda+c)\wedge \lambda$  on $(X,\infty)\cup(-\infty,-X)$. Note that $F\ge G$. Let $\gamma_G'$ be the level line of $-h-G$ from $-X$ to $X$. Since $G$ is piecewise constant, we know the curve $\gamma_G'$ is continuous and does not hit the point $\infty$. 
By Lemma \ref{lemma::absolutecontinuity}(2), since $\gamma$ is almost surely simple, we know that $\gamma$ cannot hit $(-X, \infty)$ before it hits $\infty$. Thus, we can repeat the same argument as in the proof of Lemma \ref{lem::nonboundary_intersecting} and show that $\gamma_G'$ hits $\infty$ with positive probability, contradiction. See more details in Figure \ref{fig::nohitting_target}.
\end{proof}

\begin{lemma}\label{lem::simplicity}
Suppose that $h$ is a zero boundary  $\GFF$ and that $F$ is bounded and satisfies Condition (\ref{eqn::inequalities}). Suppose that $\gamma$ is a random continuous transient curve from 0 to $\infty$ with almost surely continuous driving function. Assume that $\gamma$ is coupled with $h$ as a level line of $h+F$. Then $\gamma$ is almost surely simple. 
\end{lemma}

\begin{proof}
We can repeat the same argument as in the proof of Lemma \ref{lem::nonboundary_intersecting_simple} replacing Lemma \ref{lem::nonboundary_intersecting} by Lemmas \ref{lemma::absolutecontinuity}(1) and \ref{lemma::nothitendpoint}(1).
\end{proof}

\begin{lemma}\label{lem::mono_aux}
Suppose that  $F$ and $G$ are bounded, $F$ satisfies Condition (\ref{eqn::inequalities}), and that 
\[ F(x)\ge G(x), \quad \forall x\in\R.\]
Suppose that $\gamma_F$ (resp. $\gamma_G'$) is a random continuous transient curve from 0 to $\infty$ (resp. from $\infty$ to $0$) with almost surely continuous driving function.

Assume that $\gamma_F$ is coupled with a zero boundary  $\GFF$ $h$ as a level line of $h+F$ from 0 to $\infty$ and that $\gamma_G'$ is coupled with $h$ as a level line of $-h-G$ from $\infty$ to 0, and that the triple $(h,\gamma_F,\gamma_G')$ is coupled so that $\gamma_F$ and $\gamma_G'$ are conditionally independent given $h$. Then almost surely $\gamma_F$ stays to the left of $\gamma_G'$. 
\end{lemma}

\begin{proof}
Note that, by Lemma \ref{lem::simplicity}, $\gamma_F$ is almost surely simple. It is sufficient to show that, for any $\gamma'_G$-stopping time $\tau'$, the point $\gamma_G'(\tau')$ is to the right of $\gamma_F$. 

Let $\tilde{h}$ be $h$ restricted to $\HH\setminus \gamma_G'[0,\tau']$. Then we know that, given $\gamma_G'[0,\tau']$, the conditional law of $\tilde{h}+F$ is that of a GFF with boundary data as shown in Figure \ref{fig::generalmonotonicity}(a). Moreover, $\gamma_F$ is coupled as a level line of $\tilde{h}+F$ up until the first time it hits $\gamma'_G[0,\tau']$. Given $\gamma_G'[0,\tau']$, let $\tau$ be the first time that $\gamma_F$ hits $\gamma_G'[0,\tau']$. 

Consider the set $\gamma_G'[0,\tau']$, there are two possibilities for the intersection $\gamma_G'[0,\tau']\cap (0,\infty)$: Case (a), the intersection is nonempty, and in this case we denote by $x_G$ the last point in the intersection; Case (b), the intersection is empty and in this case we set $x_G=+\infty$ ie. to the right of $\gamma'_G[0,\tau']$.

\begin{figure}[ht!]
\begin{subfigure}[b]{0.48\textwidth}
\begin{center}
\includegraphics[width=0.73\textwidth]{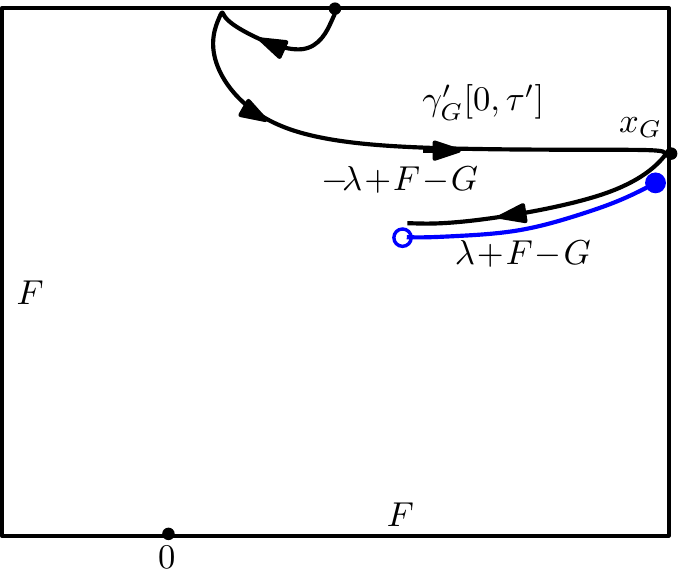}
\end{center}
\caption{}
\end{subfigure}
$\quad$
\begin{subfigure}[b]{0.48\textwidth}
\begin{center}\includegraphics[width=0.73\textwidth]{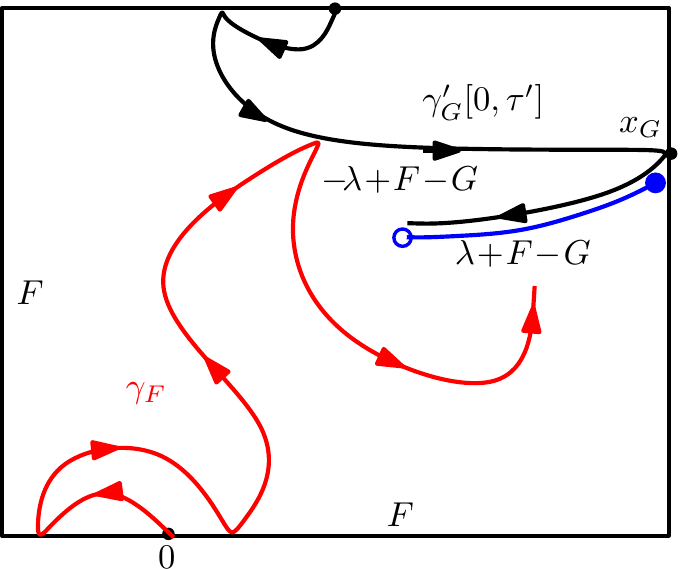}
\end{center}
\caption{}
\end{subfigure}
\caption{\label{fig::generalmonotonicity} The curve $\gamma_F$ cannot hit the blue section in the figure.}
\end{figure}

Since the boundary data on the right hand side of $\gamma_G'[0,\tau']$ is greater than $\lambda$ and the boundary data is bounded away from $-\lambda$ in a neighborhood of $x_G$, 
by Lemma \ref{lemma::absolutecontinuity} we know that $\gamma_F$ cannot hit the right hand side of $\gamma_G'[0,\tau']$ before hitting the left side of $\gamma_G'[0,\tau']$ or exiting $\HH$ at $\infty$, approaching from the left. We also know that $\gamma_F$ cannot hit $\{x_G\}$ before this time, by Lemma \ref{lemma::nothitendpoint}(1) in Case (a) and by Lemma \ref{lemma::nothitendpoint}(2) in Case (b). Therefore $\gamma_F$ cannot hit the union of the right hand side of $\gamma_G'[0,\tau']$ and $\{x_G\}$ (i.e.
the blue section in Figure \ref{fig::generalmonotonicity}(a)) of the boundary before hitting the left hand side of $\gamma_G'[0,\tau']$ or exiting $\HH$ at $\infty$, approaching from the left.
In the latter case, we are done. 
In the former case, $\gamma_F$ first hits $\gamma_G'[0,\tau']$ from its left hand side at time $\tau$.  If $\gamma_G'(\tau')$ is strictly to the left of $\gamma_F$, then it must be the case that after time $\tau$, $\gamma_F$ wraps around $\gamma_G'[0,\tau']$ and then hits the right hand side of $\gamma'_G[0,\tau']$ or exits at $x_G$. Let $\tau_\delta$ be the first time after $\tau$ that $\gamma_F$ is in the right connected component of $\HH\setminus\left(\gamma_F[0,\tau]\cup \gamma'_G[0,\tau']\right)$ and $\dist(\gamma_F(t),\gamma'_G[0,\tau])\geq \delta$, setting $\tau_\delta=\infty$ if this never happens. If $\gamma_G'(\tau')$ is strictly to the left of $\gamma_F$ (so in particular not on the curve $\gamma_F$) with positive probability then we know that $\{\tau_\delta < \infty\}$ occurs with strictly positive probability. However, given $\gamma_G'[0,\tau']\cup \gamma_F[0,\tau_\delta]$, the conditional law of $h+F$ is that of a GFF with boundary values as shown in Figure \ref{fig::generalmonotonicity}(b), and $(\gamma_F(t),t\geq \tau_\delta)$ is a level line of this field  (by  Propositions \ref{propn::localsetunions} to \ref{propn::morelocalsets}.) Therefore, by Lemmas \ref{lemma::absolutecontinuity} and \ref{lemma::nothitendpoint} again, we know that it cannot hit the right hand side of $\gamma'_G[0,\tau']$ or exit at $x_G$, and hence cannot reach $\infty$. Thus we obtain a contradiction.
\end{proof}

\begin{lemma}\label{lem::reversibility_aux}
Suppose that $F$ is bounded and satisfies Condition (\ref{eqn::inequalities}). 
Suppose that $\gamma_F$ (resp. $\gamma_F'$) is a random continuous transient curve from $0$ to $\infty$ (resp. from $\infty$ to $0$) with almost surely continuous driving functions.

Assume that $\gamma_F$ is coupled with a zero boundary  $\GFF$ $h$ as a level line of $h+F$ from 0 to $\infty$, that $\gamma_F'$ is coupled with $h$ as a level line of $-h-F$ from $\infty$ to 0, and that the triple $(h,\gamma_F,\gamma_F')$ is coupled so that $\gamma_F$ and $\gamma_F'$ are conditionally independent given $h$. Then almost surely $\gamma_F=\gamma_F'$. In particular, $\gamma_F$ is almost surely determined by $h$.
\end{lemma}

\begin{proof}
By Lemma \ref{lem::mono_aux}, we know that $\gamma_F$ almost surely stays to the left of $\gamma_F'$ and (by the same arguments) that $\gamma_F$ almost surely stays to the right of $\gamma_F'$. Combining with the fact that $\gamma_F, \gamma_F'$ are simple by Lemma \ref{lem::simplicity}, we know that  almost surely $\gamma_F=\gamma_F'$. Since $\gamma_F$ and $\gamma_F'$ are coupled with $h$ so that they are conditionally independent given $h$, $\gamma_F=\gamma_F'$ implies that $\gamma_F$ must be almost surely determined by $h$. 
\end{proof}

\begin{lemma}\label{lem::mono_continuitytransience}
Suppose that $F$ and $G$ are bounded, and satisfy Condition (\ref{eqn::inequalities}).
Suppose further that 
\[F(x)\geq G(x), \quad x\in \R.\]

Suppose that $\gamma_F, \gamma_G$ (resp. $\gamma_G'$) are random continuous transient curves from $0$ to $\infty$ (resp. from $\infty$ to 0) with almost surely continuous driving functions. 

Assume that $\gamma_F$ (resp. $\gamma_G$) is coupled with a zero boundary  $\GFF$ $h$ as a level line of $h+F$ (resp. $h+G$), that $\gamma_G'$ is coupled with $h$ as a level line of $-h-G$ from $\infty$ to 0, and that $(h,\gamma_F,\gamma_G, \gamma_G')$ is coupled so that $\gamma_F$, $\gamma_G$ and $\gamma_G'$ are conditionally independent given $h$. Then almost surely $\gamma_F$ stays to the left of $\gamma_G$.
\end{lemma}
\begin{proof}
We have the following observations.
\begin{itemize}
\item By Lemma \ref{lem::mono_aux}, we know that $\gamma_F$ stays to the left of $\gamma_G'$.  %Note that, since $G$ satisfies Condition (\ref{eqn::inequalities}) and $F\ge G$, we have $F(x)\ge -\lambda+c$ on $\R_+$.
\item By Lemma \ref{lem::reversibility_aux}, we know that $\gamma_G=\gamma_G'$.
\end{itemize}
Combining these two facts, we see that $\gamma_F$ stays to the left of $\gamma_G$.
\end{proof}

\begin{corollary}\label{prop::generalmonotonicty}
Suppose that  $F$ and $G$ are piecewise constant functions changing value only finitely many times and that they satisfy Condition (\ref{eqn::inequalities}). Suppose further that 
\[F(x)\geq G(x), \quad x\in \R.\]
Let $\gamma_F$ (resp. $\gamma_G$) be the level line of $h+F$ (resp. $h+G$) for $h$ a zero boundary $\GFF$ as in Section \ref{subsec::levelline_piecewiseconstant}. Then almost surely $\gamma_F$ stays to the left of $\gamma_G$. 
\end{corollary}
\begin{proof}
From the results in Section \ref{subsec::levelline_piecewiseconstant}, we have the existence, the continuity and transience of $\gamma_F$ and $\gamma_G$, and also $\gamma_G'$ which is the level line of $-h-G$ from $\infty$ to 0. Moreover, we know that each of $\gamma_F, \gamma_G$ and $\gamma_G'$ is almost surely determined by $h$. By Lemma \ref{lem::mono_continuitytransience}, we know that $\gamma_F$ stays to the left of $\gamma_G$ almost surely. 
\end{proof}

\section{Estimates on crossing probabilities}
\label{sec::weakcvg}

In this section, we will consider $\SLE_4(\underline{\rho}^L;\underline{\rho}^R)$ processes for vectors 
\[\underline{\rho}^L=(\rho^L_l, \cdots, \rho^L_1), \quad \quad \underline{\rho}^R=(\rho_1^R, \cdots, \rho_r^R),\]
with associated force points 
\[\underline{x}^L=(x^L_l<\cdots<x^L_1\leq 0), \quad \quad \underline{x}^R=(0\leq x^R_1<\cdots x^R_r), \] such that for some $c>0,C<\infty$,
 \begin{equation}
\label{eqn::discreteconditions} 
-2+\frac{c}{\lambda} \le \sum_{i=1}^j \rho_i^L \le -1+\frac{C}{\lambda},\quad 1\le j\le l, \quad 
 -2+\frac{c}{\lambda} \leq \sum_{i=1}^k  \rho_i^R \leq -1 + \frac{C}{\lambda},\quad 1\le k\le r.
 \end{equation} 

We will show that if $(\gamma^{(n)})_{n\in \N}$ are a family $\SLE_4(\underline{\rho}^L;\underline{\rho}^R)$ processes as above (with the same $c, C$), then they satisfy Condition \ref{cond::quadcross}. Here, we know that the processes are generated by continuous curves, due to the results of \cite{msig1, wwll1}.

Note that these processes correspond to level lines of $(h+F_n)_{n\in \N}$ for $h$ a zero boundary GFF, where Condition (\ref{eqn::discreteconditions}) means that the $F_n$'s are uniformly bounded  (lying in $(-C,C)$) and satisfy, for all $n\ge 0$,
\[F_n(x) \le \lambda - c, \quad x<0; \quad F_n(x) \ge -\lambda + c, \quad x\geq 0.\]
These are the same conditions we require on $F$ in Theorem \ref{thm::gff_levelline_coupling}. Therefore, the tactic will be to approximate such an $F$ by piecewise constant functions $F_n$ on $\R$, and show that the laws of the corresponding $\SLE_4(\underline{\rho}^{L,n}; \underline{\rho}^{R,n})$ processes converge weakly using Proposition \ref{propn::convergenceofcurves}. This limiting law will be our candidate for the level line of $h+F$. 

\begin{lemma}\label{lemma::crossingprobs}
Suppose that $(\gamma^{(n)})_{n\in \N}$ are a family of $\SLE_4(\underline{\rho}^L, \underline{\rho}^R)$ processes satisfying Condition (\ref{eqn::discreteconditions}) for some $c>0,C<\infty$ and all $n$. Then they satisfy Condition \ref{cond::quadcross}. 
\end{lemma}
\begin{proof}
Recall, we would like to show that our family satisfies a \emph{conformal bound on an unforced crossing}. That is, that there exists a constant $M>0$, such that for any of our processes $\gamma^{(n)}$, any stopping time $\tau$ and any avoidable quadrilateral of $H_\tau = \HH\setminus K_\tau^{(n)}$ whose modulus $m(Q)$ is greater than $M$, 
\[\PP\left(\gamma^{(n)}[\tau,\infty) \, \text{crosses} \, Q \; | \; \gamma^{(n)}[0,\tau]\right) \leq 1/2. \]
Here $(K_t^{(n)}, t\geq 0)$ denotes the sequence of hulls generated by $\gamma^{(n)}$.

For $n\ge 0$, the law of $\gamma^{(n)}$ is that of an $\SLE_4(\underline{\rho}^{L,n}; \underline{\rho}^{R,n})$ process with force points located at $(\underline{x}^{L,n}; \underline{x}^{R,n})$. Denote its driving function by $W^{(n)}$, its sequence of conformal mappings by $g^{(n)}$, and set $f^{(n)}=g^{(n)}-W^{(n)}$. By the results of \cite{wwll1}, we know that $\gamma^{(n)}$ can be coupled with $h$ a zero boundary GFF in $\HH$, as the level line of $h+F_n$, for
\[F_n(x) = 
\begin{cases}
-\lambda\left(1+\sum_{\{i: x_i^{L,n}\geq x\}} \rho_i^{L, n}\right), & x<0 \\
\lambda\left(1+\sum_{\{i: x_i^{R,n}\leq x\}} \; \rho_i^{R,n}\right), & x\geq 0.
\end{cases}\]
Moreover, for any stopping time $\tau$,  we know by the domain Markov property that, conditionally on $\gamma^{(n)}[0,\tau]$, the curve evolves from time $\tau$ onwards as a level line of a GFF with boundary conditions $\eta_\tau^{(n)}$ in the remaining domain. Here $\eta^{(n)}$ is defined corresponding to $F_n$ as in Definition \ref{def::gff_levelline}. The important thing to notice is that, as a result of the condition (\ref{eqn::discreteconditions}), we have 
\[\eta^{(n)}_\tau\circ (f_\tau^{(n)})^{-1} \geq - C, \quad \text{on} \; (-\infty,0) ; \quad \eta^{(n)}_\tau \circ (f_\tau^{(n)})^{-1} \geq -\lambda +c \quad \text{on}\; [0,\infty)\]
for any $\tau$ and $n$.
Therefore, if we set 
\[G_1(x) := \begin{cases}
-C, &  x<0 \\
-\lambda + c, & x\geq 0 
\end{cases}\]
we have that
\[G_1\le \eta_\tau^{(n)}\circ (f_\tau^{(n)})^{-1} ,\quad \text{on}\;\R.\]
Similarly if we set \[G_2(x) := \begin{cases} \lambda - c, & x<0 \\ C, & x\geq 0 \end{cases} \] then 
\[G_2\ge \eta_\tau^{(n)}\circ (f_\tau^{(n)})^{-1},\quad \text{on}\; \R. \]
Now, consider an avoidable topological quadrilateral $Q$ of $H_\tau$. The avoidability assumption means that, when we map it to $\HH$ via $f_\tau^{(n)}$, its image $Q'$ is a topological quadrilateral in $\HH$ as in Definition \ref{defn::topquad} with $S_1'$, $S_3'$ (the arcs touching the boundary) either both lying in $[0,\infty)$, or both in $(-\infty,0]$. 

Suppose we are in the first case. We would like to bound above the probability of $\gamma^{(n)}[\tau, \infty)$ crossing $Q$, where $Q$ has modulus greater than $M$ for some positive $M$. Equivalently, we must bound the probability of $f^{(n)}_\tau( \gamma^{(n)}[\tau,\infty))$ crossing $Q'$, noting by conformal invariance that $Q'$ also has modulus greater than $M$. If $Q'=(V', (S_k')_{1\leq k \leq 4})$, we let $Q''=(V'', S_0'', S_2'')$ be the doubly connected domain where $V''$ is the interior of the closure of $V'\cup V'*$ ($V'*$ the reflection of $V'$ in the real line) and $S_0'', S_2''$ are it's inner and outer boundary. Following the arguments in the proof of \cite[Theorem 1.10]{ksrc}, we let $x=\min(\R \cap S_0'')>0$ and $r=\max\{|z-x|: z \in S_0''\}>0$. We see that $Q''$ is a doubly connected domain separating $x$ and a point on $\partial B(x,r)$ from $0$ and $\infty$ (see Figure \ref{fig::crossing_arg}.) However, \cite[Thoerem 4.7]{ahlci} tells us that among all such domains, the one with the largest modulus (here defined as the extremal length of the curve family connecting $S_0''$ and $S_2''$ in $V''$, which satisfies $m(Q'')=m(Q')/2$) is the domain formed by removing $(-\infty,0]\cup[x,x+r]$ from the complex plane. This modulus is also calculated explicitly in \cite{ahlci} and so we may deduce that 
$$ \exp(2\pi m(Q''))\leq 16\left(\frac{x}{r}+1\right).$$ Since $$m(Q'')=\frac{m(Q)}{2}\geq \frac{M}{2},$$ this means that $r\leq \nu x$ for 
\begin{equation} \label{eqn::nu} \nu= \left( \frac{1}{16} \exp(\pi M) -1 \right)^{-1}. \end{equation} 
Note that $\nu$ can be made as small as we like by choosing $M$ large. 

\begin{figure}[ht!]
\begin{center}
\includegraphics[width=0.5\textwidth]{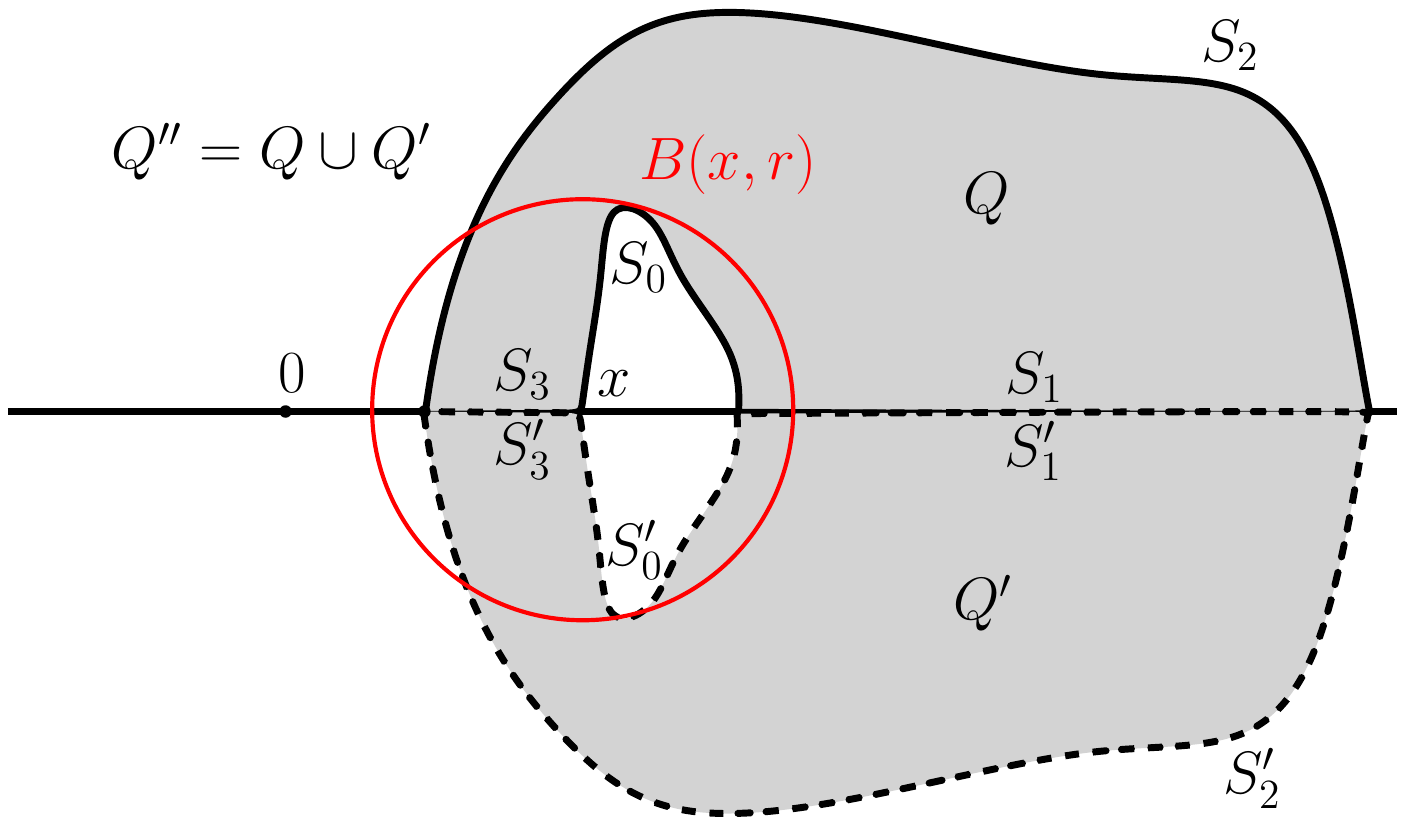}
\end{center}
\caption{\label{fig::crossing_arg} Since $Q''$ separates $x$ and a point on $\partial B(x,r)$ from $\infty$, we obtain a lower bound on $m(Q'')=m(Q)/2.$ Since $m(Q)\geq M$ this gives us an upper bound on $r/x$. Moreover, we know that for a curve to cross $Q$ it must necessarily intersect $\overline{B(x,r)}$. }
\end{figure}

It is also clear that for $f_\tau^{(n)}(\gamma^{(n)}[\tau, \infty))$ to cross $Q'$, it must necessarily intersect $\overline{B(x,r)}$. 
However, the law of $f_\tau^{(n)}(\gamma^{(n)}[\tau, \infty))$ is that of the level line of $\tilde{h}+\eta_\tau^{(n)} \circ (f_\tau^{(n)})^{-1}$, for $\tilde{h}$ a zero boundary  GFF in $\HH$. By the monotonicity result Corollary \ref{prop::generalmonotonicty}, we see that this level line lies to the left of the level line of $\tilde{h}+G_1$ almost surely (see Figure \ref{fig::crossing_prob}.) Thus, the probability of $f_\tau^{(n)}(\gamma^{(n)}[0,\tau])$ intersecting $\overline{B(x,r)}$ is less than the probability of an $\SLE_4(\rho^L;\rho^R)$ process with 
\begin{equation}
\label{eqn::uniform_rhoLrhoR}
\rho^L=-1+\frac{C}{\lambda}; \quad \rho^R = -2+\frac{c}{\lambda}
\end{equation} (left and right force points at the origin), intersecting it.

Therefore, we have
\begin{align*}
\PP&\left(\gamma^{(n)}[\tau, \infty) \; \text{crosses} \; Q \cond \gamma^{(n)}[0,\tau] \right) \\
& \leq \PP \left( f_\tau^{(n)}(\gamma^{(n)}[\tau, \infty)) \; \text{crosses} \; Q' \cond \gamma^{(n)}[0,\tau] \right) \\
& \leq  \PP \left( f_\tau^{(n)}(\gamma^{(n)}[\tau, \infty)) \; \text{intersects} \; \overline{B(x,r)} \cond \gamma^{(n)}[0,\tau] \right) \\
& \leq  \PP \left(\SLE_4(\rho^L;\rho^R) \; \text{intersects} \; \overline{B(x,r)}\right) \tag{\text{$\rho^L, \rho^R$ are defined in Equation (\ref{eqn::uniform_rhoLrhoR})}} \\
& = \PP\left(\SLE_4(\rho^L;\rho^R) \; \text{intersects} \; \overline{B(1, r/x)} \right)\tag{by scaling invariance} \\
& \leq  \PP\left(\SLE_4(\rho^L;\rho^R) \; \text{intersects} \; \overline{B(1, \nu)} \right)\tag{\text{$\nu$ is defined in Equation (\ref{eqn::nu})}}.
\end{align*}

\begin{figure}[ht!]
\begin{center}
\includegraphics[width=0.9\textwidth]{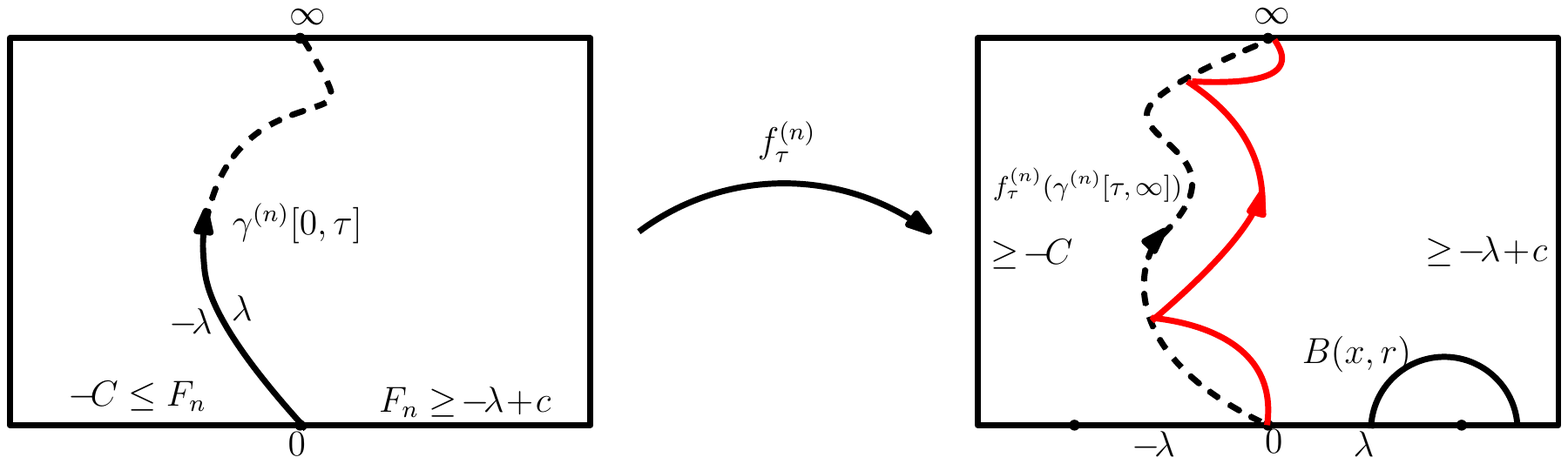}
\end{center}
\caption{\label{fig::crossing_prob} The boundary values of $h+F_n$ given $\gamma^{(n)}[0,\tau]$ are marked in the left panel. Thus $f_\tau^{(n)}(\gamma^{(n)}[\tau,\infty])$ is the level line of a zero boundary GFF $\tilde{h}$ + boundary data as depicted in the right panel. By monotonicity, it must therefore lie to the left of the level line of $\tilde{h}+G_1$ (marked in red.) Consequently, the probability that $f_\tau^{(n)}(\gamma^{(n)}[\tau,\infty])$ intersects $\overline{B(x,r)}$ is less than the probability that the red level line does.  }
\end{figure}

Since we know that the $\SLE_4(\rho^L;\rho^R)$ process with left and right force points at $0$ almost surely does not hit the point $1$ (in fact, there is exact estimate on this event, see for instance \cite[Theorem 1.8]{MillerWuSLEIntersection}), we see that by choosing $M$ large enough, and so $\nu$ small enough, we can make the right hand side less than 1/2. Thus there exists an $M$ such that the left hand side is bounded above uniformly by $1/2$ whenever $m(Q)\geq M$. 

For the second case, when the boundary arcs $S_1', S_3'$ of $Q'$ both lie on the negative real line, we may use symmetrical arguments, replacing $G_1$ by $G_2$. 
\end{proof}

\begin{corollary}\label{cor::tightness}
Suppose that $(\gamma^{(n)})_{n\in \N}$ are a family of $\SLE_4(\underline{\rho}^L; \underline{\rho}^R)$ processes satisfying Condition (\ref{eqn::discreteconditions}) for all $n$. Suppose further than they are all parameterised by half plane capacity and that $(W^{(n)})_{n\in \N}$ are the corresponding family of driving functions. Then
\begin{itemize}
\item $(W^{(n)})_{n\in \N}$ is tight in the metrisable space of continuous functions on $[0,\infty)$ with the topology of uniform convergence on compact subsets of $[0,\infty)$. 
\item $(\gamma^{(n)})_{n\in \N}$ is tight in the metrisable space of continuous functions on $[0,\infty)$ with the topology of uniform convergence on the compact subsets of $[0,\infty)$. 
\end{itemize}
Moreover, if the sequence converges weakly in either of the topologies above, then it also converges weakly in the other and the limits agree in the sense that the law of the limiting random curve is the same as the that of the random curve generated under the law of the limiting driving process.  
\end{corollary}

\begin{proof}
This is a direct consequence of Proposition \ref{propn::convergenceofcurves} and the remarks in Theorem 1.10 of \cite{ksrc}.
\end{proof}

\section{Existence of the coupling---proof of Theorem \ref{thm::gff_levelline_coupling}} 
\label{sec::existence}
In this section we will show existence of the coupling described by Theorem \ref{thm::gff_levelline_coupling}. Recall we would like to prove that for $F$ on $\R$ which is regulated, so can be approximated uniformly by piecewise constant functions changing value only finitely many times, and which satisfies Condition (\ref{eqn::inequalities}),
there exists a coupling of a Loewner chain $K$ with a zero boundary  GFF $h$, such that $K$ is a level line of $h+F$. Moreover, we will show that $K$ is almost surely generated by a continuous and transient curve $\gamma$. 
 
To do this we will take a sequence of piecewise constant functions $F_n$ (changing value only finitely many times), which uniformly approximate $F$, and consider the level lines, denoted by $\gamma^{(n)}$, of $h+F_n$ for a zero boundary GFF $h$. Observe that we can choose the $F_n$ so that the level lines are a family of $\SLE_4(\underline{\rho}^L;\underline{\rho}^R)$ processes satisfying the conditions of Corollary \ref{cor::tightness}. Thus, the tightness given by the corollary will allow us to extract a subsequential limit.

\begin{proposition}
\label{propn::weaklimit}
Let $F$ satisfy the conditions of Theorem \ref{thm::gff_levelline_coupling}. Suppose that $(F_n)_{n\in \N}$ are piecewise constant functions on $\R$, changing value only finitely many times. Let $h$ be a zero boundary  $\GFF$ and $\gamma^{(n)}$ be the level line of $h+F_n$ for each $n$. Suppose further that they are all parameterized by half plane capacity and that $(W^{(n)})_{n\in \N}$ are the corresponding family of driving functions.

Then, if the $(F_n)$ converge uniformly to $F$ on $\R$, we have that: 
\begin{enumerate}
\item [(1)] There exists a subsequence of the $\gamma^{(n)}$ which converges weakly in the space of continuous functions on $[0,\infty)$ with the topology of uniform convergence on compact subsets of $[0,\infty)$. 
\item [(2)] The limiting law describes a continuous curve from $0$ to $\infty$ in $\HH$ which generates a Loewner chain with a.s. continuous driving function.
\item [(3)] The limiting curve can be coupled with a zero boundary $\GFF$ $h$, as a level line of $h+F$.  
\end{enumerate}

\end{proposition}

\begin{remark} We will later see that this limiting law does not depend on the choice of approximation, as any continuous curve which can be coupled with a zero boundary $\GFF$ as a level line of $h+F$ must have a unique law: see Remark \ref{rmk::lawunique}. In particular, this tells us that we actually have convergence of the whole sequence in distribution.  \end{remark}

\begin{proof}[Proof of Theorem \ref{thm::gff_levelline_coupling}]
Theorem \ref{thm::gff_levelline_coupling} is a direct consequence of Proposition \ref{propn::weaklimit}.
\end{proof}
\begin{proof}[Proof of Proposition \ref{propn::weaklimit}, Items (1), (2)]
Note that the weak convergence directly follows from Corollary \ref{cor::tightness}, as does the fact that the limiting law corresponds to a continuous curve generating a Loewner chain with almost surely continuous driving function. 
\end{proof}

\begin{definition}
\label{defn::etatilde}
Suppose that $F$ is $L^1$ with respect to harmonic measure on $\R$ and that $\gamma$ is a continuous curve with continuous Loewner driving function. We set $\eta_t^0$ in the same way as in Definition \ref{def::gff_levelline}. 
Then for any $z\in \HH$ we can define, for $t$ less than the first time that $\gamma$ swallows $z$,
\[\eta_t(F, \gamma, z)=\eta_t^0(f_t(z))\]
as in Definition (\ref{def::gff_levelline}), emphasising the dependence on $F$ and $\gamma$. Let 
\[C_t(\gamma, z)=\log\CR(z, \HH)-\log \CR(z,\HH\setminus K_t),\] 
\[\tilde{\eta}_t(F, \gamma ,z)=\eta_{\tau(t)}(F, \gamma, z), \quad \text{where }\tau(t):= \inf\{s\geq 0: C_s(\gamma, z)=t\}.\]
Finally, define $(\tilde{W}_t(\gamma,z))_{t\geq 0}$ for $\gamma$ to be the driving function of $\gamma$ reparameterised by $C_t(\gamma, z)$.
\end{definition}

To prove Proposition \ref{propn::weaklimit}, Item (3), i.e. to see that the limiting curve can be coupled as a level line in the way we want, we will use Lemma \ref{lem::coupling_iff_martingale}. This tells us that if we define $\tilde{\eta}_t(F,\gamma,z)$ as above for our limiting curve $\gamma$, we need only show that for each $z\in \HH$, the process ($\tilde{\eta}_t(F,\gamma,z), t\ge 0)$ is a Brownian motion with respect to the filtration generated by $(\tilde{W}_t(\gamma,z), t\geq 0)$.

\begin{lemma}\label{lem::tildeeta_cvg_gamma}
Let $(\gamma^{(n_k)})$ be a subsequence of the random curves in Proposition \ref{propn::weaklimit}, parameterised by half plane capacity, which converge weakly to some $\gamma$ in the space of continuous functions on $[0,\infty)$ with the topology of uniform convergence on compacts. Then for every $z\in \HH$, \[\left( \tilde{W}(\gamma^{(n_k)},z), \tilde{\eta}(F, \gamma^{(n_k)},z)\right) \xrightarrow{d}\left(\tilde{W}(\gamma,z), \tilde{\eta}(F, \gamma,z)\right)\]
in $C([0,\infty);\R)\times C([0,\infty);\R)$ with respect to the product topology of uniform convergence on compacts.
\end{lemma}
 We postpone the proof of Lemma \ref{lem::tildeeta_cvg_gamma} and first tell the readers how we obtain Proposition \ref{propn::weaklimit} from Lemma \ref{lem::tildeeta_cvg_gamma}. 
\begin{lemma}
\label{lemma::etaconvergence}
Let $(\gamma^{(n_k)})$ be a subsequence of the random curves in Proposition \ref{propn::weaklimit}, parameterised by half plane capacity, which converge weakly to some $\gamma$ in the space of continuous functions on $[0,\infty)$ with the topology of uniform convergence on compacts. Then for every $z\in \HH$, \[\left( \tilde{W}(\gamma^{(n_k)},z), \tilde{\eta}(F_{n_k}, \gamma^{(n_k)},z)\right) \xrightarrow{d}\left(\tilde{W}(\gamma,z), \tilde{\eta}(F, \gamma,z)\right)\]
in $C([0,\infty);\R)\times C([0,\infty);\R)$ with respect to the product topology of uniform convergence on compacts.
\end{lemma}

\begin{proof}
By Lemma \ref{lem::tildeeta_cvg_gamma}, we have that \begin{equation}
\label{eqn::weakconvergence1}
\left(\tilde{W}(\gamma^{(n_k)},z),\tilde{\eta}(F, \gamma^{(n_k)},z)\right) \xrightarrow{d} \left(\tilde{W}(\gamma,z), \tilde{\eta}(F, \gamma,z)\right)
\end{equation} with respect to the product topology of uniform convergence on compacts. It is also clear that, for all $t,z$ and any curve $\gamma'$
\[\left|\tilde{\eta}_t(F_{n_k}, \gamma',z)-\tilde{\eta}_t(F, \gamma',z)\right| \leq \sup_{x\in \R}|F(x)-F_{n_k}(x)|.\] Indeed, $\tilde{\eta}_t(F_{n_k}, \gamma',\cdot)$ and $\tilde{\eta}_t(F, \gamma',\cdot)$ are by definition harmonic extensions of functions whose boundary values differ by at most the right hand side. Since $\sup_{x\in \R}|F(x)-F_{n_k}(x)| \to 0$ by assumption, we may conclude that, for any $T>0$, almost surely as $k\to\infty$, 
\begin{equation} \label{eqn::weakconvergence2}
\sup_{t\in [0, T]}\left|\tilde{\eta}_t(F_{n_k}, \gamma^{(n_k)},z)- \tilde{\eta}_t(F, \gamma^{(n_k)},z)\right|\to 0 .
\end{equation} 
Combining Equations (\ref{eqn::weakconvergence1}) and (\ref{eqn::weakconvergence2}), we obtain the conclusion.
\end{proof}

\begin{proof}
[Proof of Proposition \ref{propn::weaklimit}, Item (3)]
Fix $z\in \HH$. Since $\gamma^{(n_k)}$ is coupled as a level line of $h+F_{n_k}$ we know by Lemma \ref{lem::coupling_iff_martingale} that $\left(\tilde{\eta}_t(F_{n_k}, \gamma^{(n_k)},z), t\geq 0\right)$ is a Brownian motion for each $k$, with respect to the filtration of $\left(\tilde{W}_t(\gamma^{(n_k)},z), t\geq 0\right)$. Therefore, by the weak convergence in Lemma \ref{lemma::etaconvergence}, we have that if $\gamma$ is the limiting law of the $\gamma^{(n_k)}$'s, the process $\tilde{\eta}_t(F, \gamma,z)$ must also have the law of Brownian motion, with respect to the filtration of $\left(\tilde{W}_t(\gamma,z), t\geq 0\right)$.  Applying Lemma \ref{lem::coupling_iff_martingale} again proves the proposition.
\end{proof}

\begin{proof}
[Proof of Lemma \ref{lem::tildeeta_cvg_gamma}]
Fix $z\in \HH$. We will show that the laws of $(\tilde{W}(\gamma^{(n_k)},z),\tilde{\eta}(F, \gamma^{(n_k)},z))$ converge weakly in $k$ to the law of $(\tilde{W}(\gamma,z),\tilde{\eta}(F, \gamma,z))$. To do this, we begin by showing that this family of laws is tight in $C([0,\infty);\R)\times C([0,\infty);\R)$ with respect to the product topology of uniform convergence on compacts. This allows us to extract a further subsequence along which the $(\tilde{W}(\gamma^{(n_k)},z),\tilde{\eta}(F, \gamma^{(n_k)},z))$'s converge. We then argue that the limit of this subsequence must be equal to that of $(\tilde{W}(\gamma,z),\tilde{\eta}(F, \gamma,z))$, so in fact our whole original subsequence converged, and the limit is $(\tilde{W}(\gamma,z),\tilde{\eta}(F, \gamma,z))$.
Note that the proof of this lemma would be trivial if $(\tilde{W}(\cdot,z),\tilde{\eta}(F, \cdot, z))$ was a continuous function on the set of curves, however, this is not quite the case. It is essentially a continuous function when restricted to a set in which the $\gamma^{(n_k)}$'s lie with high probability. 

By the proof of \cite[Theorem 1.5]{ksrc}, we know that for every $M>0$ we can find a subset $E$ of the space of continuous curves in $\HH$ such that
\begin{equation} \label{tight} \inf_{k} \PP(\gamma^{(n_k)}\in E) \geq 1-\frac{1}{M} \end{equation}
when the $(\gamma^{(n_k)})$ are parameterised by half plane capacity, and 
\begin{itemize}
\item $E$ is relatively compact with respect to the topology of uniform convergence on compacts,
\item curves in $E$ correspond to Loewner chains with continuous driving functions parameterised by half plane capacity, and
\item if a sequence of curves in $E$ converges with respect to uniform convergence on compacts, then their driving functions also converge uniformly on compacts along a further subsequence, and the limits agree.
\end{itemize}
For the construction of such an $E$, see Section 3.5 of \cite{ksrc}, in particular the definition (60) and the discussion in the closing paragraphs. See also the opening paragraph of Section 3.6.

We argue that the set $\{(\tilde{W}(\gamma',z), \tilde{\eta}(F,\gamma',z)): \gamma' \in E\}$ is a relatively compact subset of $C([0,\infty);\R)\times C([0,\infty);\R)$ with respect to the product topology of uniform convergence on compacts. Thus by (\ref{tight}) the laws of the \[\left(\tilde{W}(\gamma^{(n_k)},z),\tilde{\eta}(F, \gamma^{(n_k)},z)\right)\] are tight in this topology. 
It is sufficient to verify the following claim: if $\gamma_n'\to \gamma'$ is any convergent sequence of curves in $E$, whose driving functions also converge uniformly on compacts, then for any $T>0$, as $n\to\infty$,
\begin{equation}\label{eqn::tildeeta_cvg_nicegamma}
\sup_{t\in [0,T]}\left| \tilde{\eta}_t(F, \gamma_n',z) - \tilde{\eta}_t(F, \gamma',z)\right| \to 0
\end{equation}
 and
\begin{equation}\label{eqn::tildew_cvg_nicegamma}
\sup_{t\in [0,T]} \left | \tilde{W}_t(\gamma_n',z)-\tilde{W}_t(\gamma',z)\right| \to 0.
\end{equation}
Relative compactness then follows because the choice of $E$ means that any sequence of curves in $E$ has a convergent subsequence along which the driving functions also converge. 

We will prove the above claim now. We let $K_t$ (resp. $K_t^n$) be the hull generated by $\gamma'$ (resp. $\gamma_n'$) in the capacity parameterisation and $W_t, g_t$ (resp. $W_t^n, g_t^n$) be the corresponding driving functions, and functions $\HH\setminus K_t$ (resp. $\HH\setminus K_t^n$) to $\HH$, normalised at $\infty$. We define $f_t=g_t-W_t$ and $f_t^n=g_t^n-W_t^n$ as usual, and consider these to be extended to the boundary, also writing $f_t(0^+)$ for $V_t^R(0^+)-W_t$. Write 
\[C_t=C_t(\gamma', z),\quad C_t^n=C_t(\gamma_n', z);\]
\[\tau(t):= \inf\{s\geq 0: C_s=t\}, \quad \tau^n(t):= \inf\{s\geq 0: C_s^n=t\}. \]

First, we will show that for any $T>0$ before the first time that $\gamma'$ swallows $z$,  as $n\to\infty$,
\begin{equation}
\label{eqn::tildeeta_cvg_nicegamma1}
\sup_{t\in[0,T]}|C_t-C_t^n| \to 0.
\end{equation}
We have the following observations.
\begin{itemize}
\item By Lemma \ref{lem::cr_evolution}, and since $C_0=C_0^n=0$, we have
\[C_t=\int_0^t \frac{-4\Im(f_s(z))^2}{|f_s(z)|^4} ds;\quad C^n_t=\int_0^t \frac{-4\Im(f^n_s(z))^2}{|f^n_s(z)|^4} ds.\]
\item $W_t^n\to W_t$ uniformly on $[0, T]$.
\item $g_t^n \to g_t$ uniformly on $\{(t,z)\in [0,T]\times \overline{\HH}: d(z,K_t)>\delta\}$ for any $\delta>0$. (See for instance \cite[Lemmas A.3 and A.4]{ksrc}) 
\end{itemize}
Combining these three facts, we obtain Equation (\ref{eqn::tildeeta_cvg_nicegamma1}).
\medbreak
Second, we show that, for any $T>0$ before $\gamma'$ swallows $z$, as $n\to\infty$, 
\begin{equation}\label{eqn::tildeeta_cvg_nicegamma2}
\sup_{t\in [0,T]}\left|\eta_{\tau^n(t)}(F, \gamma',z)-\eta_{\tau(t)}(F, \gamma',z) \right|\to 0. 
\end{equation}
By Equation (\ref{eqn::tildeeta_cvg_nicegamma1}), we have that $[0,\tau^n(T)\vee \tau(T)]\subset [0,\tau(S)]$ for $n$ large enough, where $S>T$, and $\tau(S)$ is a time before $z$ is swallowed by $\gamma'$. By (\ref{eqn::tildeeta_cvg_nicegamma1}) again, we therefore have that, as $n\to\infty$ \[c_n:=\sup_{t\in [0,T]}|C_{\tau^n(t)}-t|\le \sup_{t\in [0,\tau^n(T)\vee \tau(T)]}|C_t - C_t^n |\to 0.\]
Since
\[\sup_{t\in [0,T]}\left|\eta_{\tau^n(t)}(F, \gamma',z)-\eta_{\tau(t)}(F, \gamma',z) \right|
\le \sup_{s,t\in [0,S], |s-t|\leq c_n}\left| \tilde{\eta}_t(F, \gamma',z) - \tilde{\eta}_s(F, \gamma',z)\right|,\]
and $\tilde{\eta}_t(F, \gamma',z)$ is uniformly continuous on $[0,S]$, we see that it must converge to $0$.

\medbreak
Third, we show that, for any $T>0$ before $\gamma'$ swallows $z$, as $n\to\infty$, 
\begin{equation}\label{eqn::tildeeta_cvg_nicegamma3}
\sup_{t\in [0,T]} \left|\eta_{\tau^n(t)}(F, \gamma_n',z)-\eta_{\tau^n(t)}(F, \gamma',z) \right|\to 0. 
\end{equation}
We need only show that, on any time interval $[0,S]$ such that $S$ is strictly less than the time $\gamma'$ swallows $z$,  the quantity $\left|\eta_t(F, \gamma_n',z)-\eta_t(F, \gamma',z)\right|$ converges uniformly to $0$. We have the following observations.
\begin{itemize}
\item By Definition \ref{defn::etatilde}, we know that $\eta_t(F, \gamma', \cdot)$ (resp. $\eta_t(F, \gamma_n', \cdot)$) is the bounded harmonic function with boundary values equal $F$ on $\R\setminus K_t$ (resp. on $\R\setminus K^n_t$), $-\lambda$ on the left side of $K_t$ (resp. $K^n_t$), and $\lambda$ on the right side of $K_t$ (resp. $K^n_t$).
\item $W_t^n\to W_t$ uniformly on $[0, S]$.
\item $g_t^n \to g_t$ uniformly on $\{(t,z)\in [0,S]\times \overline{\HH}: d(z,K_t)>\delta\}$ for any $\delta>0$. Same reason as above.
\end{itemize}
Combining these three facts, we have that the quantity $\left|\eta_t(F, \gamma_n',z)-\eta_t(F, \gamma',z)\right|$ converges uniformly to $0$ on $t\in [0,S]$, implying Equation (\ref{eqn::tildeeta_cvg_nicegamma3}).
\medbreak
Combining Equations (\ref{eqn::tildeeta_cvg_nicegamma1}), (\ref{eqn::tildeeta_cvg_nicegamma2}) and (\ref{eqn::tildeeta_cvg_nicegamma3}), we obtain Equation (\ref{eqn::tildeeta_cvg_nicegamma}) by noting that 
\[\sup_{t\in [0,T]}\left| \eta_{\tau^n(t)}(F, \gamma_n',z)-\eta_{\tau(t)}(F, \gamma',z) \right| \leq \sup_{t\in [0,T]} \left|\eta_{\tau^n(t)}(F, \gamma_n',z)-\eta_{\tau^n(t)}(F, \gamma',z) \right| + \sup_{t\in [0,T]}\left|\eta_{\tau^n(t)}(F, \gamma',z)-\eta_{\tau(t)}(F, \gamma',z) \right|.\]
\medbreak
We obtain Equation (\ref{eqn::tildew_cvg_nicegamma}) by the same method as above, which is much simpler in this case, and so we omit the details.
\medbreak 
Finally, we show that if $(\gamma^{(n_k)})_{k\in \N}$ converges weakly, and there exists a further subsequence along which $(\tilde{W}(\gamma^{(n_k)},z), \tilde{\eta}(F, \gamma^{(n_k)},z))$ converges, then the limit must be $(\tilde{W}(\gamma,z),\tilde{\eta}(F, \gamma,z))$.
To do this, for any $M\in \N$ take $E$ relatively compact such that (\ref{tight}) holds, and note that by the above claim we have that 
\[A_E:=\left\{(\gamma', \tilde{W}(\gamma',z), \tilde{\eta}(F, \gamma',z)): \gamma' \in E\right\}\] is relatively compact in $C([0,\infty);\C)\times C([0,\infty);\R)\times C([0,\infty);\R)$ , and its closure is equal to 
\[\left\{(\gamma', \tilde{W}(\gamma',z),\tilde{\eta}(F, \gamma',z)): \gamma' \in \overline{E}\right\}.\]
This means that the joint laws of $(\gamma^{(n_k)}, \tilde{W}(\gamma^{(n_k)},z),\tilde{\eta}(F,\gamma^{(n_k)},z))$ are also tight, and thus we can extract an even further subsequence along which we have joint convergence. If $\PP^*$ is the law of this joint limit then, 
\[ \PP^*\left(\overline{A_E}\right)\geq \inf_k \PP\left( \gamma^{(n_k)} \in \overline{E} \right) \geq 1- \frac{1}{M} \] and so we see that the 
probability of our marginal laws agreeing in the sense we want must be greater than $1-\frac{1}{M}$. Since this holds for every $M$, agreement must hold almost surely, and as these marginal laws are equal to the limiting laws of the individually convergent sequences, the result follows. 
 
\end{proof}

\section{Proof of Theorems \ref{thm::gff_levelline_determination} to \ref{thm::reversibility}}
\label{sec::reversibilityetc}
\begin{proof}[Proof of Theorem \ref{thm::monotonicity}]
Suppose that $\gamma_F$ and $\gamma_G$ are continuous transient curves from $0$ to $\infty$ in $\HH$, coupled with a zero-boundary GFF $h$ as level lines of $h+F$ and $h+G$ respectively. Suppose further that $\gamma_G'$ is a continuous transient curve from $\infty$ to $0$ and is coupled with $h$ as a level line of $-h-G$ from $\infty$ to $0$, such that the four objects $h, \gamma_F, \gamma_G, \gamma_G'$ are coupled with $\gamma_F, \gamma_G, \gamma_G'$ are conditionally independent given $h$. 
From Theorem \ref{thm::gff_levelline_coupling}, we have the existence of $\gamma_F, \gamma_G$ and $\gamma_G'$. By Lemma \ref{lem::mono_continuitytransience}, we know that $\gamma_F$ stays to the left of $\gamma_G$ almost surely.
\end{proof}

\begin{proof}[Proof of Theorems \ref{thm::gff_levelline_determination} and \ref{thm::reversibility}]

Suppose that $\gamma_F$ is a continuous transient curve which is coupled with $h$ as a level line of $h+F$ from $0$ to $\infty$, as in Theorem \ref{thm::gff_levelline_coupling}. Let $\gamma_F'$ be a continuous curve coupled with $h$ as a level line of $-h-F$ from $\infty$ to $0$, such that $\gamma_F$ and $\gamma_F'$ are conditionally independent given $h$. The existence of $\gamma_F'$ is given by Theorem \ref{thm::gff_levelline_coupling}. Lemma \ref{lem::reversibility_aux} then tells us that $\gamma_F=\gamma_F'$ almost surely. In particular, $\gamma_F$ is almost surely determined by $h$. 
\end{proof}

\begin{remark} \label{rmk::lawunique}
By applying Theorem \ref{thm::gff_levelline_determination}, we see that if $\gamma$ is the weak limit of any sequence of level lines as in Proposition \ref{propn::weaklimit}, then $\gamma$ can be coupled as the level line of a $\GFF$ and is moreover determined by the $\GFF$ in this coupling. Thus, the law of $\gamma$ is uniquely determined. In particular, it does not depend on the sequence of approximating level lines. 
\end{remark}

\begin{lemma}\label{lem::mono_cvg}
Let $F$ be as in Theorem \ref{thm::gff_levelline_coupling}. Suppose that $F_n\downarrow F$ approximate $F$ uniformly on the real line, where the $F_n$ are decreasing, and are piecewise constant with value changing only finitely many times. 

Let $h$ be a zero boundary  $\GFF$ in $\HH$, $\gamma_n$ be the level line of $h+F_n$ for each $n$, and $\gamma$ be the level line of $h+F$. Denote by 
$H_n$ the open sets corresponding to the strict right hand sides of $\gamma_n$.  By monotonicity these are almost surely decreasing. Define 
\[H= \cap_n  \overline{H_n}.\]
Then $\partial H$ coincides with $\gamma$ almost surely. In other words, the sequence of curves $\gamma_n$ converges to $\gamma$ almost surely.
\end{lemma}
\begin{proof}
First, we show that $\partial H$ has the same law as $\gamma_F$.
We use a conformal mapping to take everything to the unit disc, as it will be more convenient to work in a space where our sets are compact. We endow $\HH$ with the metric it inherits from the unit disc $\U$ via the map $\varphi(z)=(z-i)/(z+i)$. Namely, let $d_*(\cdot, \cdot)$ denote the metric on $\HH$ given by 
\[d_*(z,w)=|\varphi(z)-\varphi(w)|.\] 
We write $\overline{\HH}$ for the completion of $\HH$ with respect to $d_*$. 
For compact sets $A, B\subset\overline{\HH}$, we have the $d_*$-induced Hausdorff distance 
\[d_*^H(A, B)=\inf\{\eps>0: A\subset B^{(\eps)}, B\subset A^{(\eps)}\},\]
where $A^{(\eps)}$ denotes the open $\eps$-neighborhood of $A$ with respect to the metric $d_*$. Note that $d_*^H$ makes the set of all non-empty compact subsets of $\overline{\HH}$ (with metric $d_*$) into a compact metric space. 
We have the following observations.
\begin{itemize}
\item The sets $\overline{ H_n}$ form an almost surely decreasing sequence of compact subsets of $\overline{\HH}$, which therefore converge to $H$ with respect to $d_*^H$. This implies that $\gamma_n=\partial H_n$ almost surely converges to $\partial H$ with respect to $d_*^H$.
\item By the assumptions on $F_n$, we know that the laws of $\gamma_n$ fall in to the framework of Proposition \ref{propn::weaklimit}. This means that we can extract a subsequence which converges weakly in the space of continuous functions on $[0,\infty)$ with respect to uniform convergence on compacts. Moreover, the limiting curve can be coupled with a zero boundary $\GFF \, \tilde{h}$ as the level line of $\tilde{h}+F$. Furthermore, the subsequence converges weakly, to the same limit, in the space of curves from $[0,1]\to \overline{\HH}$ with respect to the topology of uniform convergence modulo reparameterisation, where the metric on $\overline{\HH}$ is given by $d_*$. This requires a slight extension of Proposition \ref{propn::convergenceofcurves}, which was stated here, but is nonetheless still true, by the extended version given in \cite[Corollary 1.7]{ksrc}. By continuity, we therefore have that along this subsequence the curves also converge weakly to the same limit with respect to $d_*^{H}$. Thus $\partial H$ has the law of a continuous curve which can be coupled with a zero boundary $\GFF \, \tilde{h}$ as the level line of $\tilde{h}+F$.
%\item By the assumptions on $F_n$, we know that the laws of $\gamma_n$ fall in to the framework of Proposition \ref{propn::convergenceofcurves}. This means that we can extract a subsequence which converges weakly in the space of continuous curves with respect to the topology of uniform convergence among all reparameterisations, as described in \cite{ksrc}. Note that this does rely on a slight variant of Proposition \ref{propn::convergenceofcurves}, which was stated here, but is nonetheless still true, by the extended version given in \cite[Theorem 1.5]{ksrc}. Furthermore, by continuity, we have that along this subsequence the curves converge weakly with respect to $d_*^{H}$. Moreover, the limiting curve can be coupled with a $\GFF \tilde{h}$ as the level line of $\tilde{h}+F$.
\item By Theorem \ref{thm::gff_levelline_determination}, we know that the law on continuous curves which can be coupled with a $\GFF \tilde{h}$ as a level line of $\tilde{h}+F$, is unique.
\end{itemize}
Combining these three facts, we may conclude that $\partial H$ has the same law as $\gamma_F$. 
\medbreak
Next, we show that $\partial H$ coincides with $\gamma$ almost surely. We have the following observations.
\begin{itemize}
\item By the above analysis, we know that $\partial H$ has the same law as $\gamma$.
\item By Theorem \ref{thm::monotonicity}, we know that $\partial H$ lies to the left of $\gamma$ almost surely. \end{itemize}
Combining these two facts, we obtain that $\partial H$ coincides with $\gamma$ almost surely.
\end{proof}

\section{Proof of Theorem \ref{thm::sle4rho_existence}  and concluding remarks}
\label{sec::sle4rho}
In this section, we prove Theorem \ref{thm::sle4rho_existence}: the key ingredient being the proof of Lemma \ref{lem::sle4_martingalecharacterisation}. This lemma is proved in \cite{msig1, wwll1} for $\SLE_4(\rho)$ process when $\rho$ is a vector. The proof given in these papers will work with minor modifications for the case when $\rho$ is a Radon measure but, to be self-contained, we still give a complete proof here.

\begin{lemma}\label{lem::sle4_martingalecharacterisation}
Suppose we are given a random continuous curve in $\overline{\HH}$ from 0 to $\infty$ whose Loewner driving function $W$ is almost surely continuous. If $(\rho^L; \rho^R)$ are a pair of finite Radon measures on $\R_-,\R_+$ and $F$ is the corresponding function of bounded variation, define $(\eta_t, t\geq 0)$ as in Definition \ref{def::gff_levelline}. For $z\in\HH$ and $t\ge 0$, define 
\[\tau(t)=\inf\{s: \log\CR(z, \HH)-\log\CR(z, \HH\setminus K_s)=t\}.\]
Then $\left(W, (V^L(x))_{x\in \R_-}, (V^R(x))_{x\in \R_+}\right)$ can be coupled with a standard Brownian motion to describe an $\SLE_4(\rho^L; \rho^R)$ process if $(\eta_{\tau(t)}(z), t\ge 0)$ evolves as a 
Brownian motion with respect to the filtration generated by $(W_{\tau(t)}, t\ge 0)$ for any $z\in \HH$.
\end{lemma}

\begin{proof}
Suppose that $(\eta_{\tau(t)}(z), t\ge 0)$ is a Brownian motion with respect to the filtration generated by $(W_{\tau(t)}, t\ge 0)$ for each $z\in \HH$. This implies that $(\eta_t(z), t\ge 0)$ is a local martingale with respect to the filtration generated by $(W_t, t\ge 0)$. Our first step will be to show that $W_t$ is a continuous semi-martingale. 
By the definition of $\eta_t(\cdot)$, we know that, for each $z\in\HH$,
\begin{align}
\label{eqn::etarep}
2\eta_t(z) = &  - \int_{\R_-} \arg(g_t(z)-V_t^L(x)) \, \rho^L(dx) - \arg(g_t(z)-W_t) \nonumber \\ 
& +  \left(\pi - \arg(g_t(z)-W_t)\right) + \int_{\R_+} \left( \pi - \arg(g_t(z)-V_t^R(x)) \right) \, \rho^R(dx).
\end{align}
This follows from the integration by parts formula for functions of bounded variation, and the integral expression for the harmonic extension of a bounded function on the real line. Note here that the integrals are well defined, since for each fixed $t,z$ the integrands are continuous, bounded functions in $x$, and $\rho^L, \rho^R$ are assumed to be finite measures. 
Indeed, $g_t(z)$ and $(V_t^{L,R}(x))_{x\in\R}$  are adapted and differentiable, and we may also differentiate under the integral in (\ref{eqn::etarep}) by finiteness of $\rho^L, \rho^R$. Therefore, we can deduce that all the terms in (\ref{eqn::etarep}) apart from the only one, $\arg(g_t(z)-W_t)$, involving $W_t$, are semi-martingales. Since $\eta_t(z)$ is itself a local martingale, this means that $\arg(g_t(z)-W_t)$ must also be a semi-martingale. Now, note that by Schwartz's formula, we can write $\log(g_t(z)-W_t)$, up to a constant, as a linear functional (an integral against a test function) of $\arg(g_t(z)-W_t)$. So $\log(g_t(z)-W_t)$ is also a semi-martingale, and thus it's exponential, and consequently $W_t$ itself, must be a semi-martingale also.
Hence we can write $W_t:=M_t-V_t$
for $M$ a local martingale and $V$ of bounded variation. 

Substituting this into the expression (\ref{eqn::etarep}) we see that, on intervals where $W_t$ does not collide with the $V_t^{L,R}$, the drift of $2\eta_t$ is equal to the imaginary part of 
\[ \frac{2 \int_{\R_-} \rho^L(dx)/\left(V_t^L(x)-W_t\right) }{g_t(z)-W_t} dt + \frac{-2dV_t}{(g_t(z)-W_t)}+ \frac{d\langle W_t \rangle - 4 dt}{(g_t(z)-W_t)^2}+ \frac{2 \int_{\R_+} \rho^R(dx)/\left(V_t^R(x)-W_t\right)}{g_t(z)-W_t} dt, \]
which of course must vanish. Therefore, multiplying by $(g_t(z)-W_t)^2$ and evaluating at $z$ such that $g_t(z)-W_t$ is arbitrarily close to 0, we can deduce that $d\langle W_t \rangle =4dt$. On subsequently removing the third term, we also find an expression for $dV_t$, and can conclude that $W_t$ satisfies (\ref{wtsde}) in Definition \ref{defn::slekapparho} on intervals where $W_t$ does not collide with the $V_t^{L,R}$.

 All that remains is to show that we have instantaneous reflection of $W_t$ off the $V_t^{L,R}(x)$. It suffices to show that the number of times the curve $\gamma$ hits the real line has Lebesgue measure 0. However, this is always the case for a continuous curve with continuous driving function, which we know for example by \cite[Lemma 2.5]{msig1}. 

\end{proof}

\begin{remark}
 We believe that Lemma \ref{lem::sle4_martingalecharacterisation} could be made into an if and only if statement if we strengthened Definition \ref{defn::slekapparho} of an $SLE_\kappa(\rho^L; \rho^R)$ process to also require that, almost surely, 
 \begin{equation}\label{eqn::forcepoints_nopush}
 V_t^L(x)= x + \int_0^t \frac{2 ds}{V_s^L(x) - W_s},\quad x\in\R_-;\quad 
 V_t^R(x)= x + \int_0^t \frac{2 ds}{V_s^R(x) - W_s},\quad x\in\R_+ \; \;\; \text{and}
 \end{equation}
 \begin{equation}\label{eqn::driving_nopush}
 W_t=\sqrt{\kappa}B_t+\int_0^t ds \int_{\R_-} \frac{\rho^L(dx)}{W_s-V_s^L(x)}+\int_0^tds \int_{\R_+} \frac{\rho^R(dx)}{W_s-V_s^R(x)},
 \end{equation}
as in \cite{wwll1} and \cite{msig1}.
 
 That is, using the stronger definition, we could show that any such process can always be coupled with the Gaussian Free Field as generalized level line. This would also give us uniqueness in law for the $\SLE_4(\rho^L;\rho^R)$ process among continuous curves. However, it seems that (\ref{eqn::forcepoints_nopush}) and (\ref{eqn::driving_nopush}) are hard to verify assuming only that $(\eta_{\tau(t)}(z), t\geq 0)$ evolves as a Brownian motion. 
\end{remark}

\begin{proof}
[Proof of Theorem \ref{thm::sle4rho_existence}.] Combining Theorem \ref{thm::gff_levelline_coupling} with Lemmas  \ref{lem::coupling_iff_martingale} and \ref{lem::sle4_martingalecharacterisation} in the case that $F$ is of bounded variation, we know that in the coupling $(h,\gamma)$ given by Theorem \ref{thm::gff_levelline_coupling}, the marginal law of $\gamma$ is that of an $\SLE_4(\rho^L;\rho^R)$ process. This gives us existence of the process. Moreover, we know the curve $\gamma$ is almost surely continuous and transient and also satisfies the reversibility property (3) of Theorem \ref{thm::sle4rho_existence}, by Theorem \ref{thm::reversibility}.

%For uniqueness, suppose there is a law on continuous transient curves satisfying the definition of an $\SLE_4(\rho^L;\rho^R)$ process. Then we know by Lemmas  \ref{lem::coupling_iff_martingale} and \ref{lem::sle4_martingalecharacterisation} that the curve can be coupled as the level line of $h+F$ for $h$ a zero boundary GFF as in Definition \ref{def::gff_levelline}. In this case we see by Theorem \ref{thm::gff_levelline_determination} that the curve is almost surely determined by $h$; in particular, its law is uniquely determined. 
\end{proof}

\begin{remark}
We can also generalize the construction of flow lines and counterflow lines to $\GFF$ with general boundary data. 
A similar approximating idea works for flow lines and counterflow lines. Since the flow lines and counterflow lines have a \textit{duality property}, instead of reversibility as for level line case, some  extra work is neeeded for the proof of monotonicity as in Section \ref{sec::mono}. The details are left to interested readers.
\end{remark}
\begin{remark}
As explained in the introduction, we restrict to boundary values satisfying Condition (\ref{eqn::inequalities}) throughout the paper. This condition guarantees that there is no continuation threshold. The continuity of the level lines when there does exist a continuation threshold is still open. 
\end{remark}

\bibliographystyle{alpha}
\bibliography{gff_bibliography}

\begin{thebibliography}{Zha08b}

\bibitem[Ahl73]{ahlci}
L.~V. Ahlfors.
\newblock {\em Conformal invariants: topics in geometric function theory}.
\newblock McGraw-Hill Book Co., New York, 1973.

\bibitem[Die69]{jd}
Jean Dieudonn\'{e}.
\newblock {\em Foundations of modern analysis}.
\newblock Academic Press, New York, 1969.

\bibitem[Dub09]{dub}
Julien Dub\'{e}dat.
\newblock {S}{L}{E} and the free field: partition functions and couplings.
\newblock {\em J. Amer. Math. Soc.}, 22(4):995--1054, 2009.

\bibitem[Fol99]{realanalysis}
Gerald~B. Folland.
\newblock {\em Real analysis: modern techniques and their applications}.
\newblock John Wiley and {S}ons inc., 2nd edition, 1999.

\bibitem[KI13]{hadamard}
Kalle Kyt\"{o}l\"{a} and Konstantin Izyurov.
\newblock Hadamard's formula and couplings of {SLE}s with free field.
\newblock {\em Probability Thoery and Related Fields}, 155(1-2):35--69, 2013.

\bibitem[KS16]{ksrc}
Antti Kemppainen and Stanislav Smirnov.
\newblock Random curves, scaling limits and {L}oewner evolutions.
\newblock {\em Annals of Probability}, page to appear, 2016.

\bibitem[MS16a]{msig1}
Jason Miller and Scott Sheffield.
\newblock Imaginary geometry {I}: interacting {SLE}s.
\newblock {\em Probab. Theory Related Fields}, 164(3-4):553--705, 2016.

\bibitem[MS16b]{MillerSheffieldIG3}
Jason Miller and Scott Sheffield.
\newblock Imaginary geometry {III}: reversibility of {SLE}$_{\kappa}$ for
  $\kappa$ in (4, 8).
\newblock {\em Ann. of Math.}, 184(2):455--486, 2016.

\bibitem[MW16]{MillerWuSLEIntersection}
Jason Miller and Hao Wu.
\newblock Intersections of {SLE} paths: the double and cut point dimension of
  {SLE}.
\newblock {\em Probability Theory and Related Fields}, page to appear, 2016.

\bibitem[RS05]{rs}
Steffen Rohde and Oded Schramm.
\newblock Basic properties of {SLE}.
\newblock {\em Ann. of Math. (2)}, 161(2):883--924, 2005.

\bibitem[SS09]{ss09}
Oded Schramm and Scott Sheffield.
\newblock Contour lines of the two-dimensional discrete {G}aussian free field.
\newblock {\em Acta Math.}, 202(1):21--137, 2009.

\bibitem[SS13]{ss}
Oded Schramm and Scott Sheffield.
\newblock A contour line of the continuum {G}aussian free field.
\newblock {\em Probab. Theory Related Fields}, 157(1-2):47--80, 2013.

\bibitem[WW13]{WernerWuCLESLE}
Wendelin Werner and Hao Wu.
\newblock From {CLE}($\kappa$) to {SLE}($\kappa$, $\rho$)’s.
\newblock {\em Electron. J. Probab}, 18(36):1--20, 2013.

\bibitem[WW16]{wwll1}
Menglu Wang and Hao Wu.
\newblock Level {L}ines of {G}aussian {F}ree {F}ield {I}: Zero-boundary {GFF}.
\newblock {\em Stochastic Processes and their Applications}, page to appear,
  2016.

\bibitem[Zha08a]{ZhanDualityChordalSLE}
Dapeng Zhan.
\newblock Duality of chordal {SLE}.
\newblock {\em Inventiones mathematicae}, 174(2):309--353, 2008.

\bibitem[Zha08b]{ZhanReversibilityChordalSLE}
Dapeng Zhan.
\newblock Reversibility of chordal {SLE}.
\newblock {\em The Annals of Probability}, 36(4):1472--1494, 2008.

\end{thebibliography}
\bigbreak
\noindent Ellen Powell\\
\noindent Department of Pure Mathematics and Mathematical Statistics\\
\noindent University of Cambridge, Cambridge, England\\
\noindent ep361@cam.ac.uk
\bigbreak
\noindent Hao Wu\\
\noindent NCCR/SwissMAP, Universit\'{e} de Gen\`{e}ve, Switzerland\\
\noindent \textit{and }Yau Mathematical Sciences Center, Tsinghua University, China\\
\noindent hao.wu.proba@gmail.com

\end{document}